\documentclass[11pt,a4paper]{article}

\usepackage{amsfonts,amsmath,layout}
\usepackage{mathrsfs}  
\usepackage{amssymb,mathabx,amsfonts,amsthm,amscd,stmaryrd,dsfont,esint,upgreek,constants,todonotes}

\usepackage[utf8]{inputenc}
\usepackage{makeidx}
\usepackage[T1]{fontenc} 
\usepackage{amsmath}
\usepackage{amsthm}
\usepackage{amsfonts}
\usepackage{amssymb}
\usepackage{alltt}
\usepackage{color}
\usepackage{graphicx}

\newcommand{\R}{\mathbb R}

\def\dk{{\bf d}_1}
\def\J{{\mathcal J}}
\newcommand{\be}{\begin{equation}}
\newcommand{\ee}{\end{equation}}
\def\1{{\bf 1}}

\def\ep{\epsilon}

\def\b{\beta}
\def\Dt0{{\bf D}(t_0)}

\def\E{{\bf E}}

\def\to{\rightarrow}

\def\ds{\displaystyle}

\def\a{\alpha}
\def\b{\beta}
\def\g{\gamma}

\def\l{\lambda}

\def\ds{\displaystyle}

\def \to{\rightarrow}

\def \R {\mathbb{R}}

\def\dk{{\bf d}_1}
\def\dive{{\rm div}}

\def \ep{\varepsilon}

\def\E{\mathbb E}

\definecolor{ProcessBlue}{cmyk}{1,0,0,0.40}

\newtheorem{Theorem}{Theorem}[section]

\newtheorem{Proposition}[Theorem]{Proposition}
\newtheorem{Lemma}[Theorem]{Lemma}

\newtheorem{Remark}{Remark}[section]

\newtheorem{Example}{Example}[section]
\begin{document}
\title{Convergence of some  Mean Field Games systems\\to aggregation and flocking models}
\author{Martino Bardi\thanks{Department of Mathematics ``T. Levi-Civita'', University of Padova, Via Trieste 63, 35121 Padova, Italy - bardi@math.unipd.it}
 and Pierre Cardaliaguet\thanks{Universit\'e Paris-Dauphine, PSL Research University, Ceremade, 
Place du Mar\'echal de Lattre de Tassigny, 75775 Paris cedex 16 - France - cardaliaguet@ceremade.dauphine.fr }}
\date{\today}
\maketitle
\abstract{For two classes of Mean Field Game systems we study the convergence of solutions as the interest rate in the cost functional becomes very large, modeling agents caring only about a very short time-horizon, and the cost of the control becomes very cheap. The limit in both cases is a single  first order integro-partial differential equation for the evolution of the mass density. The first model is a 2nd order MFG system with vanishing viscosity, and the limit is an aggregation equation. The result has an interpretation for models of collective animal behaviour and of crowd dynamics. The second class of problems are 1st order MFGs of acceleration and the limit is the kinetic equation associated to the Cucker-Smale model. The first problem is analyzed by PDE methods, whereas the second is studied by variational methods in the space of probability measures on trajectories.}


 \tableofcontents

\section*{Introduction}

The aim of this work is to discuss, in some particular settings, how models involving crowds of rational agents 
 continuous in space-time can degenerate to  agent based models as the agents become less and less rational. The models of rational agents used in this paper are the  Mean Field Games 
  (MFG), introduced by Lasry and Lions \cite{LL07mf} (see also \cite{HMC}). They describe optimal control problems with infinitely many infinitesimal agents who interact through their distribution. 

Our results  are
 inspired on one hand by the last part of \cite{LLB}, in which the authors show how to derive a McKean-Vlasov equation from a mean field game system and, on the other hand, by \cite{DHL}  (see also \cite{Ba}) which discusses how 
  multi-agent control problems in which the players have limiting anticipation converge to aggregation models. 
Let us briefly recall the content of both papers. In \cite{LLB}, the authors study MFG systems 
 of the form 
\be\label{modelBLL}
\left\{ \begin{array}{l}
-\partial_t u_\l -\nu\Delta u_\l +H(x,Du_\l,m_\l(t)) +\lambda u= 0\; {\rm in }\; \R^d\times (0,T), \\
\partial_t m_\l-\nu\Delta m_\l -{\rm div}(m_\l D_pH(x,Du_\l,m_\l(t))=0\; {\rm in }\; \R^d\times (0,T), \\
u_\l(T,x)= u_T(x), \; m_\l(0)=m_0  \;{\rm in}\; \R^d. 
\end{array}\right.
\ee
Here  $\nu>0$ is fixed, $\lambda>0$ is a large parameter which describes the impatience of the players and $H=H(x,p,m)$ is the Hamiltonian of the problem which includes interaction terms between the players. Under suitable assumptions on the data, \cite{LLB} states that,  as $\lambda$ tends to $\infty$ and up to subsequences, $u_\l, Du_\l \to 0$ and $m_\l$ converges 
 to a solution of the  McKean-Vlasov equation
$$
\left\{ \begin{array}{l}
\partial_t m-\nu\Delta m -{\rm div}(mD_pH(x,0,m(t))=0\qquad {\rm in }\; \R^d\times (0,T), \\
 m(0)=m_0 \qquad {\rm in}\; \R^d.
\end{array}\right.
$$
Possible variants and extensions (to MFG models with relative running costs and to higher order approximation) are also discussed in \cite{LLB}. 

Although \cite{DHL} shares some common features with \cite{LLB}, it  is quite different. It proposes a  deterministic model in which, as in \cite{LLB}, the agents have little rationality, in the sense that they anticipate only on a short horizon (here through time discretization). On the other hand, and this is in contrast with \cite{LLB}, the agents are supposed to pay little for their move. The paper \cite{DHL} explains, at least at a heuristic level, that the optimal feedback control of each agent should converge to the gradient descent of the running cost, which the authors call
 ``Best Reply Strategy''. They also discuss the limit of the distribution of agents as their number 
  goes to infinity and the related $1$st order
 McKean-Vlasov equation.

In the present paper we consider a continuous time variant of the model in \cite{DHL} which contains its two main features: the fact that the players 
minimize a cost on a very short horizon, that we model  as in \cite{LLB} by a large discount factor, and the fact that they pay little for their moves.  
To fit also better with aggregation or  kinetic models, we  work with problems with a vanishing viscosity ($\nu=\nu_\l\to 0^+$ as $\l\to+\infty$)
 and in infinite horizon. In particular, our result makes rigorous the approach of \cite{DHL}.
 
We prove two convergence results. In the first one, our  model (in its simplest version) takes the form 
\be\label{mfgIntro}
\left\{ \begin{array}{l}
\ds -\partial_t u_\lambda -\nu_\l\Delta u_\l+\lambda u_\lambda   +\frac\l 2 |Du_\lambda|^2 = F(x,m_\lambda(t)) \qquad {\rm in }\; \R^d\times (0,+\infty)
\\
\ds \partial_t m_\lambda-\nu_\l\Delta m_\l -\dive(m_\lambda
\l Du_\lambda 
)=0 \qquad {\rm in }\; \R^d\times (0,+\infty)\\
\ds m_\lambda(0)=m_0,\qquad {\rm in }\;  \R^d, \qquad u_\lambda\; {\rm bounded} .
\end{array}\right.
\ee
Under  some natural assumptions on $F$ (typically, continuous on $\R^d\times {\mathcal P}_2(\R^d)$ and uniformly Lipschitz continuous and semi-concave in the space variable), we show that, as $\lambda$ tends to infinity (meaning that players become more and more myopic and that their control is increasingly cheap) and along subsequences, $m_\l$ converges to a solution $m$ of the 
aggregation model 
\be
\label{aggrIntro}
\left\{ \begin{array}{l}
\ds \partial_t m -\dive(m D_xF(x,m)
)=0 \qquad {\rm in }\; \R^d\times (0,+\infty)\\
\ds m(0)=m_0,\qquad {\rm in }\;  \R^d .
\end{array}\right.
\ee
Moreover, the optimal feedback $-\l Du_\l$ for the generic agent in the MFG converges a.e. to the vector field $-D_xF(\cdot,m)$, giving the gradient descent of the running cost corresponding to  the limit distribution of agents $m$.
To compare with \cite{LLB}, let us note that, in the case where $H(x,p,m)= \frac12 |p|^2-F(x,m)$, the limit of \eqref{modelBLL} is a simple diffusion equation, while in our setting this limit is non trivial. 

The limit equation in \eqref{aggrIntro} covers most examples of the so-called Aggregation Equation
\[
\partial_t m + \dive\left(m \int_{\R^d}K(x-y)m(y) \, dy\right)=0 ,
\]
because the kernel of the convolution is usually the gradient of a potential, $K=-D k$. This equation describes the collective behavior of various animal 
 populations, its derivation and the choice of the kernel are based on phenomenological considerations, see, e.g., \cite{TBL, BT} and the references therein. In Subsection \ref{ex and app} we show that the examples of Aggregation Equation most studied  in the mathematical biology literature fit the assumptions of our convergence theorem, as well as some known models of crowd dyanmics. Therefore our result gives a further justification of such models within the framework of dynamic games with a large number of players.

Our second result concerns (first order) MFG of acceleration \cite{AMMT,CM}, formally written in the form
 \be\label{MFGaIntro}
 \left\{ \begin{array}{l}
\ds -\partial_t u_\lambda+\lambda u_\lambda -v\cdot D_xu_\l +\frac\l 2 |D_vu_\lambda|^2 = F(x,v, m_\lambda(t)) \qquad {\rm in }\; (0,+\infty)\times \R^{2d}
\\
\ds \partial_t m_\lambda + v\cdot D_xm_\l-\dive_v(m_\lambda 
\l D_v u_\lambda
)=0 \qquad {\rm in }\; (0,+\infty)\times \R^{2d}\\
\ds m_\lambda(0)=m_0,\qquad {\rm in }\;  \R^{2d} .
\end{array}\right.
\ee
 for which we prove the convergence to kinetic equations of the form
\be\label{CSIntro}
\left\{ \begin{array}{l}
\ds \partial_t m  + v\cdot D_xm -\dive_v(m D_vF(x,v, m))=0 \qquad {\rm in }\; (0,+\infty)\times \R^{2d},\\
\ds m(0)=m_0,\qquad {\rm in }\;  \R^{2d},
\end{array}\right.
\ee
as $\lambda \to+\infty$. To fix the ideas we work in the case where the coupling term $F$ corresponds to the Cucker-Smale model:  
$$
F(x,v,m(t)) = k * m(x,v,t) = \int_{\R^{2d}}  k(x-y, v-v_*)m(y,v_*,t) dy dv_* ,
$$
where 
$$
k(x,v)=\frac{|v|^2}{(\a +|x|^2)^\b}, \quad  \a>0 , \beta \geq 0
$$
and $v$ represents the velocity of the generic agent. Note that, in contrast with the first result, the coupling function $F$ is no longer globally Lipschitz continuous: as we explain below, this is a source of major issues and it obliges us to change completely the analysis.

Let us briefly explain the mechanism of proofs and the differences with the existing literature. In \cite{LLB}, the rough idea is that $u_\l
$ converges to $0$ and therefore $Du_\l$ converges to $0$ as well. In addition, the fact that the diffusion is nondegenerate ($\nu>0$) provides $C^{2+\alpha, 1+\alpha/2}$ bounds on $u_\l$ and $m_\l$, thanks to which one can pass  to the limit. 

For our first result, (Theorem \ref{theo_main1}, on the convergence of \eqref{mfgIntro} to \eqref{aggrIntro}),  we have to use a different argument. The key idea is that  $\l u_\l$ behaves like $F(x,m_\l)$, because $\l^{-1}F(x,m_\l)$ is almost a solution to \eqref{mfgIntro}. Therefore $\l D u_\l$  is close to $D F(x,m_\l)$, which explains the limit equation \eqref{aggrIntro}. Compared to \cite{LLB}, an additional difficulty comes from the lack of (uniform in $\l$) smoothness of the solutions, since we have no diffusion term in the limit equation. In particular, the product $m_\l \l Du_\l$ has to be handled with care, since $m_\l$ could degenerate as a measure while $Du_\l$ could become singular. We overcome this issue by proving a uniform semi-concavity of $\l u_\l$, which provides at the same time the $L^1_{loc}$ convergence of  $\l Du_\l$ and, thanks to an argument  going
 back to \cite{LL07mf} (see also \cite{CarH}) a (locally in time) uniform $L^\infty$ bound on the density of $m_\l$, and hence a weak-* convergence of $m_\l$. 

For the second result  (Theorem \ref{thm.mainBIS}, on the convergence of \eqref{MFGaIntro} to \eqref{CSIntro}), the fact that the coupling function $F$ growths in a quadratic way with respect to the (moment of) the measure prevents us 
from using  fixed point techniques (as in \cite{AMMT, CM}) to show the existence of a solution to the MFG system \eqref{MFGaIntro} and to obtain estimates on the solution (this would also be the case in the presence of a viscous term). This  obliges us to give up the PDE approach of the previous set-up and to use variational techniques, first suggested for MFG problems in \cite{LL07mf} 
 and largely developed since then: see, for instance, among many other contributions, \cite{BeCaSa, Ca15, CaGr15, CaMeSa, OrPoSa}. 
For that very same reason, we have to work with a finite horizon problem and with initial measure having a compact support. In contrast with the first result, we do not prove the convergence of {\it all} the solutions of the MFG system, but only for the ones which minimize the energy  written formally as
\be\label{lkajnzrd}
\int_0^Te^{-\l t} \int_{\R^{2d}}  (\frac{1}{2\l} |\alpha(x,v,t)|^2 + \int_{\R^{2d}} k(x-x_*,v-v_*)m(dx_*,dv_*,t)) m(dx,dv,t)dt
\ee
where $\partial_t m +v\cdot D_xm +{\rm div}_v( m\alpha)=0$.  We formulate this problem  in the space of  probability measures on curves $(\g, \dot\g)$, and the main technique of proof consists in obtaining estimates on the solution based on the dynamic programming principle in such  space.
This is reminiscent of ideas developed in \cite{RSSS} that we discuss below. Such an approach naturally involves weak solution of the MFG system and does not require the initial measure $m_0$ to be absolutely continuous. In this case the natural notion of solution for the limit equation \eqref{limitFPBIS} is the measure-valued solution developed in \cite{CCR} for \eqref{CSIntro}.

We could also have developed this second approach for the first type of results (i.e., the convergence of \eqref{mfgIntro} to \eqref{aggrIntro}), assuming that  the coupling function $F$ derives from an energy (the so-called potential mean field games): 
$$
F(x,m)= \frac{\delta F}{\delta m}(m,x)
$$
(see \cite{AGS} for the notion of derivative). Then it is known \cite{LL07mf} that minimizers $( m_\l,  \alpha_\l)$ of the problem 
\be\label{ailqkjsndc}
\inf\left\{ \int_0^{+\infty}e^{-\l t}\left( \int_{\R^d} \frac12 |\alpha|^2dx+ \l {\mathcal F}(m(t)) \right)dt, \qquad \partial_tm+{\rm div}(m\alpha)=0, \; m(0)=m_0\right\}, 
\ee
are solutions to the MFG system \eqref{mfgIntro} (with $\nu_\l =0$ and if ${\mathcal F}$ is smooth enough)  in the sense that there exists $ u_\l$ such that $( u_\l,  m_\l)$ solves \eqref{mfgIntro} and $ \alpha_\l= -\l D u_\l$. The convergence of minimizers, as $\l\to+\infty$, is studied in the nice paper \cite{RSSS}, where this convergence  is called ``Weighted Energy-Dissipation'': the authors prove that, under suitable assumptions on the function ${\mathcal F}$ (which allow for singular coupling functions), minimizers converge to a solution of the gradient flow associated to ${\mathcal F}$, i.e., at least at a formal level, to a solution of \eqref{aggrIntro}. Let us note that, in contrast with our setting, the solution of the limit equation can be singular and that \cite{RSSS} works in general metric spaces. It would be interesting to understand the precise interpretation of the results of \cite{RSSS} in terms of limits of MFGs, but this exceeds the scope of the present 
 paper.  Note however that our second result (i.e., the convergence  of \eqref{MFGaIntro} to \eqref{CSIntro}) does not 
 fit in the framework of \cite{CCR}. Indeed, the key idea of \cite{CCR} is that $m_\l$ is a gradient flow for the value function associated with Problem \eqref{ailqkjsndc}; as this value function converges to ${\mathcal F}$, $(m_\l)$ has to converge to the gradient flow for ${\mathcal F}$, which is precisely $m$; this gradient flow structure is completely lost in our framework of MFG of acceleration \eqref{MFGaIntro}: we have therefore to design a different approach. 

\subsection*{Notation} For any $p\geq 1$ we denote by ${\mathcal P}_p(\R^d)$ (or, in short ${\mathcal P}_p$) the set of Borel probability measures with finite $p-$order moment $M_p$: 
$$
M_p(m):= \int_{\R^d} |x|^p m(dx).
$$
The sets ${\mathcal P}_p(\R^d)$ are endowed with the corresponding  Wasserstein distance. 
Given a positive constant $\kappa$, we denote by ${\mathcal M}_{p,\kappa}(\R^d)$ the set of measures $m\in {\mathcal P}_p(\R^d)$ absolutely continuous with respect to the Lebesgue measure and with a density 
 bounded by $\kappa$.  We set ${\mathcal M}_{p}(\R^d) := \bigcup_{\kappa>0} {\mathcal M}_{p,\kappa}(\R^d)$. 
  In Section \ref{basic2BIS} we will also use, for $m\in{\mathcal P}_2(\R^d\times \R^d)$,
 $$
M_{2,v}(m):= \int_{\R^{2d}} |v|^2 m(dx, dv).
$$

\subsection*{Acknowledgment} The first-named author is member of the Gruppo Nazionale per l'Analisi Matematica, la Probabilit\`a e le loro Applicazioni (GNAMPA) of the Istituto Nazionale di Alta Matematica (INdAM); he was partially supported by the research project  ``Nonlinear Partial Differential Equations: Asymptotic Problems and Mean-Field Games" of the Fondazione CaRiPaRo. The second author was partially supported by the ANR (Agence Nationale de la Recherche) project  ANR-12-BS01-0008-01, by the CNRS through the PRC grant 1611 and by the Air Force Office for Scientific Research grant FA9550-18-1-0494. This work started during the visit of the second author at Padova University: the University is warmly thanked for this hospitality.

\section{Convergence for classical MFG systems} \label{sec.classicMFG}

In this section we consider MFG systems of the form  
\be\label{mfglamBIS}
\left\{ \begin{array}{l}
\ds -\partial_t u_\lambda-\nu_\l\Delta u_\l+\lambda u_\lambda  + \l^{-1} H(\l Du_\lambda,x) = F(x,m_\lambda(t)) \qquad {\rm in }\; \R^d\times (0,+\infty)
\\
\ds \partial_t m_\lambda-\nu_\l\Delta m_\l -\dive(m_\l
D_pH(\l Du_\l,x)
)=0 \qquad {\rm in }\; \R^d\times (0,+\infty)\\
\ds m_\lambda(0)=m_0,\qquad {\rm in }\;  \R^d, 
\end{array}\right.
\ee
where $\l>0$,  $\nu_\l>0$ and $\nu_\l\to 0$ as $\l \to +\infty$.
Our aim is to prove the convergence (up to a subsequence) of $m_\l$ as $\l\to +\infty$ to a solution $m$ of 
\be
\label{limitFP1}
\left\{ \begin{array}{l}
\ds \partial_t m -\dive(m D_pH(D_xF(x, m(t)),x))=0 \qquad {\rm in }\; \R^d\times (0,+\infty),\\
\ds m(0)=m_0,\qquad {\rm in }\;  \R^d.
\end{array}\right.
\ee
and to show also that
\[
\l u_\l(x,t) \to F(x,m(t))  \quad \text{loc. uniformly}, \qquad \l D u_\l(x,t) \to D_xF(x,m(t))  \quad \text{a.e.}
\]

\subsection{The convergence results}

 We work under the following conditions:  we assume that the initial measure  $m_0$ satisfies
\be 
\label{assm_0}
m_0 \in {\mathcal P}_2(\R^d)  \text{ is  absolutely continuous with a bounded density. }
\ee
The kind of costs we are interested in are non-local and regularizing. 
A possible assumptions on $F$ is that $F : \R^d\times {\mathcal M}_{1} (\R^d) \to \R$ is continuous in a suitable topology, has a linear growth and is Lipschitz continuous and  semi-concave in $x$. More precisely, we suppose the existence of a constant $C_o\geq1$ such that:
\begin{align} 
\label{assF0}
\text{For any $\kappa>0$, the restrictions of $F$ and $D_xF$ to   $\R^d\times {\mathcal M}_{1,\kappa} (\R^d)$}\notag \\
\text{ are continuous in both variables for the topology of $\R^d\times {\mathcal P}_{1} (\R^d)$}, 
\end{align}
\be
\label{assF}
|F(x,m)|\leq C_o(1+|x|), \qquad |F(x,m)-F(y,m)|\leq C_o|x-y|,
\ee
\be
\label{assFsc}
F(x+h,m)+F(x-h,m)-2F(x,m) \leq C_o |h|^2, \quad  \forall \, m \in {\mathcal M}_{1} (\R^d),
\ee
(recall that ${\mathcal M}_{p,\kappa}(\R^d)$ and ${\mathcal M}_{p}(\R^d)$ are defined in the introduction). 
We assume that $H:\R^d\times \R^d\to \R$ is convex with respect to the first variable and satisfies, 
 \be\label{HypH}
-C_o \leq H(p,x) \leq C_o(1+|p|^2) ,\qquad D^2_{pp}H(p,x)\geq C_o^{-1}I_d, 
\ee
\be\label{Hyp2}
|H(p,x)-H(p,y)|+|D_pH(p,x)-D_pH(p,y)| \leq C_o |x-y|(1+|p|),
\ee
\be\label{Hyp3}
|H(p,x)-H(q,x)|\leq C_o |p-q|(1+|p|+|q|),
\ee
\be\label{Hyp4}
H(p,x+h)+H(p,x-h)-2H(p,x) \geq -C_o |h|^2(1+|p|). 
\ee
Note that, if $H$ is smooth, then   conditions \eqref{Hyp2}, \eqref{Hyp3} and \eqref{Hyp4} can be equivalently rewritten as 
$$
|D_xH(p,x)|+|D_{px}H(p,x)| \leq C_o (1+|p|),
$$
$$
|D_pH(p,x)|\leq C_o (1+|p|),\qquad 
D^2_{xx}H(p,x) \geq -C_o (1+|p|). 
$$

\begin{Theorem}
\label{theo_main1}
Assume  \eqref{assm_0}, \eqref{assF0}, \eqref{assF},  \eqref{assFsc}, \eqref{HypH},  \eqref{Hyp2},  \eqref{Hyp3} and  \eqref{Hyp4}. Let $(u_\l, m_\l)$ be a solution to \eqref{mfglamBIS}. 
Then $(m_\l)$ is relatively compact in $C^0([0,T], {\mathcal P}_1(\R^d))$ and is bounded in $L^ \infty(\R^{d}\times [0,T])$ for any $T>0$. Moreover, the limit $m$,  as $\l_n\to +\infty$, of any converging subsequence $(m_{\l_n})$ in $C^0([0,T], {\mathcal P}_1(\R^d))$ is a solution of \eqref{limitFP1} in the sense of distributions and
\[
\l_n u_{\l_n}(x,t) \to F(x,m(t))  \quad \text{locally uniformly and} \quad  \l_n D u_{\l_n}(x,t) \to 
D F(x,m(t))  \quad \text{a.e.}
\]
\end{Theorem}


The existence of a solution to \eqref{mfglamBIS} under the assumptions above can be established by standard arguments, using  the estimates in Section \ref{sec.prooftheo1} below,  Remark \ref{existence}. A typical example of a Hamiltonian satisfying our assumptions in 
$$
H(x,p)= -v(x)\cdot p+\frac12|p|^2, 
$$
where the vector field $v:\R^d\to \R^d$ is bounded and with bounded first and second order derivatives. 

\begin{Remark}
\label{deterministic}
\upshape
In the case of deterministic  MFGs, $\nu_\l=0$ for all $\l$, the solution $(u_\l, m_\l)$ is not smooth and the proof of convergence by PDE methods is harder. We can prove a result analogous to Theorem \ref{theo_main1} under the additional assumption that $\|F(\cdot,m)\|_{C^2}\leq C$ for all $m \in {\mathcal M}_{1} (\R^d)$, and the support of  $m_0$ is compact, using the methods of \cite{Car}. In this  case we can also prove that $m_\l(t)$ has a support uniformly bounded for $t\in[0,T]$: a result of this kind is proved in Section \ref{basic2BIS} for the MFGs of acceleration. Then we expect uniqueness in the limit equation: see the next remark.
\end{Remark}
\begin{Remark}
\label{cor:uniq}
\upshape
In addition to the assumptions of Theorem \ref{theo_main1} suppose that the vector field $G$ appearing in the  limit equation \eqref{limitFP1}, $G(x,m):=-D_pH(D_xF(x,m),x)$, is such that, for all $m \in {\mathcal M}_{1} (\R^d)$, $x\mapsto G(x,m)$ is $C^1$ and 
\[
|G(x,m) - G(y,m)|\leq C_1|x-y|\,,\quad \|G(\cdot,m) - G(\cdot,\bar m)\|_\infty\leq C_1 \dk(m,\bar m) \,,
\]
where $\dk$ is the 1-Wasserstein distance. Then it is proved in \cite{PR1} that there is a unique solution $m$ of \eqref{limitFP1} with compact support in $x$. Therefore, under these additional conditions, the whole family $m_\l$ converges to $m$ as $\l\to +\infty$, as in the problem of Section \ref{basic2BIS}. For instance, the support of $m$ is compact in $x$ in deterministic MFGs, if $m_0$ has compact support, see the preceding remark. 
\end{Remark}
\begin{Remark}
\label{viscous}
\upshape
The case of $\nu_\l\to \nu_\infty >0$ can be treated as in the proof of Theorem \ref{theo_main1} and leads in the limit to the viscous Fokker-Planck equation
\[
 \partial_t m - \nu_\infty\Delta m - \dive(m D_pH(D_xF(x, m(t)),x))=0 \qquad {\rm in }\; \R^d\times (0,+\infty).
\]
\end{Remark}

\subsection{Proof of Theorem \ref{theo_main1}}\label{sec.prooftheo1}

In this part, assumptions  \eqref{assm_0}, \eqref{assF0}, \eqref{assF}, \eqref{assFsc}, \eqref{HypH},  \eqref{Hyp2},  \eqref{Hyp3} and  \eqref{Hyp4} are in force. 
We start with some estimates for a solution to  \eqref{mfglamBIS}.

\begin{Proposition}\label{prop.1} Let $(u_\l,m_\l)$ be a solution of \eqref{mfglamBIS}. Then $\ds | u_\l(x,t) | \leq \l^{-1}\tilde C(1+|x|)$ for some constant $\tilde C$ independent of $\l \geq 1+4d+16C_0^2$ such that $\nu_\l\leq 1$.
\end{Proposition}

\begin{proof} We 
 note that $w^\pm(x,t):= \pm \l^{-1} \tilde C(1+|x|^2)^{1/2}$ is a supersolution (for +) and a subsolution (for -) of \eqref{mfglamBIS} for a suitable $\tilde C$.  Let us determine $C$ such that $w= - \l^{-1} C(1+|x|^2)^{1/2}$ is a subsolution, the other case being easier.  By the growth assumptions \eqref{assF} and \eqref{HypH}, and for $\nu_\l\leq 1$,
\begin{align*}
 & -\partial_t w-\nu_\l\Delta w+\lambda w  + \l^{-1} H(\l Dw,x) - F(x,m_\lambda(t)) { \leq }\\
 & \quad { \frac{C d\nu_\l}{\l(1+|x|^2)^{1/2}}   }- C(1+|x|^2)^{1/2} + \frac 1\l H\left(\frac{ - Cx}{(1+|x|^2)^{1/2}}, x\right) - F(x,m_\lambda(t))\leq\\
 & \qquad { \frac{Cd}\l  - \frac{C}{2}(1+|x|) }+ \frac {C_0}\l \left(1+\frac{C^2|x|^2}{(1+|x|^2)}\right) + C_0(1+|x|) .
\end{align*}
If we choose ${C=4C_0}$ the right hand side can be bounded above by
\[
{\frac{4 C_0d}{\l} } - C_0 +C_0 \frac{1+16 C_0^2}\l \leq 0
\]
if $ \l \geq 1+4d+16C_0^2$.
\end{proof}

\begin{Proposition}\label{prop.2} Let $(u_\l,m_\l)$ be a solution of \eqref{mfglamBIS}. Then $\ds \| Du_\l \|_\infty \leq 4\l^{-1}C_o$ for $\l\geq 2C_o$.
\end{Proposition}

\begin{proof} We use an a priori estimate, proving that, if $u_\l$ is Lipschitz continuous and if $H$  and $(x,t)\to F(x,m_\l(t))$ are smooth, then $u_\l$ satisfies the required estimate. One can then complete the proof easily, approximating the HJ equation by HJ equations with smooth and globally Lipschitz continuous Hamiltonians and right-hand sides and passing to the limit. We omit this last part which is standard and proceed with the argument. 

Given a direction $\xi\in \R^d$ with $|\xi|\leq 1$, let $w:= Du_\l\cdot \xi$. Then $w$ satisfies 
$$
-\partial_t w -\nu_\l \Delta w  +\l w + D_pH(\l Du_\l, x) \cdot Dw +\l^{-1} D_\xi H(\l Du_\l,x)= D_\xi F(x,m(t)).
$$
In view of assumptions \eqref{assF} and  \eqref{Hyp2} we have therefore
$$
-\partial_t  w -\nu_\l \Delta  w +\l w+ D_pH(\l Du_\l, x) \cdot Dw  -C_o \l^{-1} (1+\l \|Du_\l\|_\infty) \leq C_o.
$$
So by the maximum principle we have 
$$
Du_\l \cdot\xi = w\leq \l^{-1} C_o ( 1+\l^{-1}+  \|Du_\l\|_\infty). 
$$
Taking the supremum over $|\xi|\leq 1$,  gives the result for $\l$ larger than $2C_o$. 
\end{proof}

\begin{Proposition}\label{prop.3} Let $(u_\l,m_\l)$ be a solution of \eqref{mfglamBIS}. Then $D^2 u_\l \leq \l^{-1}\tilde C$, where $\tilde C$ does not depend on $\l\geq 2C_o$.
\end{Proposition}

\begin{proof} Here again we  focus on a priori estimates for smooth data. Given a direction $\xi\in \R^d$ with $|\xi|\leq 1$, let $w:= Du_\l \cdot \xi$
and $z:= D^2u_\l\xi \cdot \xi$. Then 
\begin{align*}
&-\partial_t z -\nu_\l \Delta  z +\l z+ D_pH(\l Du_\l, x) \cdot Dz +2 D_{p,\xi} H(\l Du_\l, x) \cdot Dw \\
&\qquad +\l D_{pp} H(\l Du_\l,x)Dw\cdot Dw
+ \l^{-1}D_{\xi\xi} H(\l Du_\l,x) = D^2_{\xi\xi}F(x,m_\l(t)).
\end{align*}
Since  $F$ is semiconcave in $x$ \eqref{assFsc}, the right hand side $D^2_{\xi\xi}F$ is bounded above by $C_0$.
The uniform bound on $\l Du_\l$ proved in Proposition \ref{prop.2} and the assumption \eqref{Hyp2} 
imply
\[
2 D_{p,\xi} H(\l Du_\l, x) \cdot Dw \geq -2C_0(1+\l |Du_\l|) |Du_\l| \geq -C_1 .
\]
The same bound and the assumption \eqref{Hyp4} 
imply
\[
\l^{-1}D_{\xi\xi} H(\l Du_\l,x)\geq -C_0(1+\l |Du_\l|) \geq -C_2.
\]
Since $H$ is
convex in $p$,  
$D^2_{pp}H\geq 0$ and we infer that $z$ satisfies 
 \begin{align*}
&-\partial_t z -\nu_\l \Delta  z +\l z+ D_pH(\l Du_\l, x) \cdot Dz  \leq \tilde C, 
\end{align*}
where the constant $\tilde C$ does not depend on $\l$ and $|\xi|\leq 1$. We 
conclude again by the maximum principle. 
\end{proof}

\begin{Proposition}\label{prop.mlcompact} Let $(u_\l,m_\l)$ be a solution of \eqref{mfglamBIS}. For any $T>0$, the family $(m_\l)$ satisfies
$$
\sup_{\lambda\geq 2C_o} \sup_{t\in [0,T]} \int_{\R^d} |x|^2m_\l(x,t)dx <+\infty,  
$$
is relatively compact in $C^0([0,T], {\mathcal P}_1)$ and bounded in $L^\infty(\R^d\times [0,T])$. 
\end{Proposition}

\begin{proof} We do the proof again for smooth data. 
For the bound on the second order moment of $m_\l(t)$ on $[0,T]$ we recall that $m_\l(t)$ is the law  $\mathcal L(X_t)$ of the solution $X_t$ of the SDE
\[
dX_t = - D_pH(\l Du_\l(X_t),X_t) dt + \sqrt{2 \nu_\l} dW_t , \qquad \mathcal L(X_0)=m_0 ,
\]
where $W_t$ is a standard Brownian motion. Since the vector field $D_pH(\l Du_\l,x)$ is uniformly bounded by Proposition \ref{prop.2}, we have $\E [|X_t|^2] \leq C(\E[ |X_0|^2] + 1)e^{Ct}$. Then
\[
M_2( m_\l(t))= \int_{\R^d} |x|^2m_\l(x,t)dx = \E |X_t|^2 \leq C(M_2( m_0) + 1)e^{CT} , \forall t\leq T .
\]
For the $L^\infty$ bound on $m_\l$, we rewrite the equation of $m_\l$ as 
\begin{align*}
& \partial_t m_\lambda-\nu_\l\Delta m_\l -m_\l
{\rm Tr} \left(D_{pp}H(\l Du_\l,x)D^2 u_\l+ D_{px}H(\l Du_\l,x) \right)\\
& \qquad  -Dm_\l \cdot D_pH(Du_\l,x)
=0 \qquad {\rm in }\; \R^d\times (0,+\infty)
\end{align*}
where, by convexity of $H$, \eqref{HypH} and  Proposition \ref{prop.3} on the one hand, and  by \eqref{Hyp2} and Proposition \ref{prop.2} on the other hand, we have 
$$
{\rm Tr} \left(D_{pp}H(\l Du_\l,x)D^2 u_\l\right) \leq C\; \text{and}\;  {\rm Tr} \left(D_{px}H(\l Du_\l,x)\right)\leq C,
$$
where $C$ does not depend on 
$\l$. Therefore, by the maximum principle again, the 
$L^\infty$ norm of $m_\l$ has at most an exponential growth in time, uniform with respect to $\l$. 
 \end{proof}

 \begin{proof}[Proof of Theorem \ref{theo_main1}] By Proposition \ref{prop.mlcompact}, $(m_\l)$ is relatively compact in $C^0([0,T], {\mathcal P}_1(\R^d))$ and is bounded in $L^ \infty(\R^{d}\times [0,T])$ for any $T>0$.
Let $(m_{\l_n})$ be a converging subsequence in $C^0([0,T],{\mathcal P}_1)$ for any $T>0$. Then  $(m_{\l_n})$ converges to $m$ in $L^\infty-$weak-* on $\R^d\times [0,T]$ for any $T>0$. In particular, by our continuity assumption on $F$ in \eqref{assF0}, the maps $(x,t)\to F(x,m_{\l_n}(t))$ and $(x,t)\to D_xF(x,m_{\l_n}(t))$ converge locally uniformly to the maps  $(x,t)\to F(x,m(t))$ and $(x,t)\to D_xF(x,m(t))$ respectively. 

As $u_\l$ solves \eqref{mfglamBIS}, $w_\l:= \l u_\l$ solves 
$$
-\l^{-1}\partial_t w_\lambda-\l^{-1}\nu_\l \Delta w_\l+w_\lambda +\l^{-1}H(Dw_\l,x)=   F(x, 
 m_\lambda(t)) \qquad {\rm in }\; \R^{d}\times (0,+\infty).
$$
Hence the half-relaxed limits $w^*$ and $w_*$ of $(w_\l)$ (which are locally uniformly bounded in view of Proposition \ref{prop.1}) are respectively sub- and super-solutions of the trivial equation 
$$
w= F(x, m(t)) \qquad {\rm in }\; \R^{d}\times (0,+\infty).
$$
This proves the locally uniform convergence of $(\l_n u_{\l_n})$ to $F(x, m)$. 

Next we use Theorem 3.3.3 in \cite{CS}. By  Proposition \ref{prop.1} $(\l u_\l)$ is uniformly locally bounded,  and by Proposition \ref{prop.3} it is uniformly semi-concave in space (locally in time). Then any sequence  $(\l_n u_{\l_n})$ has a subsequence such that $(\l_n Du_{\l_n})$ converges to $D_xF(x, m)$ a.e. and therefore also in $L^1_{loc}(\R^{d}\times [0,+\infty))$. 
One easily derives from this that $m$ solves \eqref{limitFP1} in the sense of distribution.
 \end{proof}

\begin{Remark}
\label{existence}
\upshape    
The existence of a solution $(u_\l, m_\l)$ of the system \eqref{mfglamBIS} can be proved by approximating with solutions of the following system with finite time-horizon
\be\label{mfglamT}
\left\{ \begin{array}{l}
\ds -\partial_t u^T-\nu_\l\Delta u^T+\lambda u^T  + \l^{-1} H(\l Du^T,x) = F(x,m^T(t)) \quad {\rm in }\; \R^d\times (0,T)
\\
\ds \partial_t m^T-\nu_\l\Delta m^Tl -\dive(m^T
D_pH(\l Du^T,x)
)=0 \qquad {\rm in }\; \R^d\times (0,T)\\
\ds u^T(T)=0, \qquad m^T(0)=m_0, \qquad {\rm in }\;  \R^d .
\end{array}\right.
\ee
The existence of a solution  $(u^T, m^T)$ for fixed $\l >0$ {follows from standard argument (see for instance Lions' course of Nov. 12, 2010  \cite{LiCourse})}.  
 The estimates of Propositions \ref{prop.1}, \ref{prop.2}, \ref{prop.3}, and \ref{prop.mlcompact} hold for $(u^T, m^T)$ with the same proof (using comparison principles for Cauchy problems with constant terminal data). Then there is enough compactness to pass to the limit as $T\to +\infty$, as in the proof of Theorem \ref{theo_main1}, and see that the limit satisfies \eqref{mfglamBIS}.
\end{Remark}

\subsection{Examples}
\label{ex and app}

In this section we present several examples of coupling functions $F$ 
 of the form 
\be
\label{F}
F(x,m) = k * m(x,t) = \int_{\R^d}  k(x-y)m(dy) ,
\ee
where the convolution kernel $k$ can take different forms and is at least globally Lipschitz continuous. 
$$
H(p)=|p|^2/2-v(x)\cdot p ,
$$
with the vector field $v$ bounded together with its first and second derivatives.
Then 
$$
D_pH(D_xF(x,m))= \int_{\R^d}  Dk(x-y)m(dy) - v(x) ,
$$
and the limit equation \eqref{limitFP1} becomes
\begin{equation}
\label{AE}
\left\{ \begin{array}{l}
\ds \partial_t m +     \dive\left(m ( v -Q[m])\right)=0 \; {\rm in }\; \R^d\times (0,+\infty),\;  Q[m](x,t) =  \int_{\R^d} D k(x-y)m(y,t) dy ,\\
\ds m(0)=m_0,\qquad {\rm in }\;  \R^d.
\end{array}\right.
\end{equation}
 Note that the condition \eqref{assF0} is satisfied. 
In addition, we suppose that $k$ is bounded, which implies condition \eqref{assF}, and the semi-concavity of $k$, which ensure 
 condition \eqref{assFsc}.  Under these assumptions Theorem \ref{theo_main1} holds. Next we review some special cases that arise in applications.

\subsubsection{The aggregation equation}
\label{aggreg}

 The special case of \eqref{AE} with $v\equiv 0$ is often called the aggregation equation. For suitable choices of the kernel $k$ it models the collective behaviour of groups of animals, see, e.g., \cite{BV06, TBL} and the references therein.
Most kernels used in the aggregation models are of the form $k(x)= \phi
(|x|)$ with $\phi$ smooth but $\phi'(0)$ not necessarily 0, so 
$k$ can be not differentiable in the origin.
However, most of them satisfy the assumptions above.

\begin{Example} \label{ex:exp} 
\upshape
The kernel 
 \be\label{k1}
 k(x) = \a e^{-a|x|} , \qquad 
  a>0,
 \ee
considered in \cite{BV06, TBL} (see also the references therein), is bounded, globally Lipschitz continuous, and semiconcave  if $\a >0$. 
Note that the case $\a >0$ describes repulsion among individuals at all distances, because $k(x)=\phi 
(|x|)$ and $\phi'(r)<0$ implies repulsion.
The case $\a <0$, describing attraction, does not fit into our theory because  $k\sim |x|$ near 0, so it is not semiconcave, which is consistent with the fact that solutions of the aggregation equation \eqref{AE} are known to blow up in finite time for suitable initial data (at least in dimension $d=1$, see \cite{BV06}).
\end{Example}

\begin{Example}
\upshape
The kernel 
 \be\label{k2}
 k(x) = -|x| e^{-a|x|} , \qquad a>0,
 \ee
considered in \cite{BV06} 
 is also bounded, globally Lipschitz continuous and semiconcave because $k\sim -|x|$ near 0. Note that this kernel describes repulsion at small distance and attraction at distance $|x|>1/a$. Our theory is consistent with the global existence of solutions of the aggregation equation \eqref{AE} in this case, at least for $d=1$, proved in \cite{BV06}.
\end{Example}
\begin{Example}\label{ex:morse} 
\upshape
To model repulsion at short distance and attraction at medium range, decaying at infinite, a  commonly used kernel is the so-called Morse potential
 \be\label{k3}
 k(x) = e^{-|x|} - G e^{-|x|/L}    , \qquad 0<G<1, \quad L>1,
 \ee
  see \cite{BT} and the references therein. It is again bounded and globally Lipschitz continuous. It is  also semiconcave because $k\sim 1 - G +|x|(G/L-1)$ near 0, and  $G/L-1<0$.   
  \end{Example}

\subsubsection{Models of crowd dynamics}
\label{crowd}

There is a large and fast growing literature on
 models of the interactions among pedestrians, see the survey in the book \cite{CPTbook}.  They split into first order models, where the velocity of the pedestrian is a prescribed function of the density of individuals and position, and second order models, 
 where the acceleration is prescribed. In the next Section \ref{basic2BIS} we study second order models, focusing on the celebrated Cucker-Smale model of flocking, see Remark \ref{crowd2} for more references on crowd dynamics. 
 
A first order model fitting in the assumptions of the present section is 
 the one proposed in \cite{CPT11}, where the velocity of each agent at position $x$ and time $t$
is of the form $v(x) - Q[m(t)](x)$, 
 $v$ being the desired velocity of the pedestrian, and the other term $Q$ accounting for  the interaction with the other agents. If we assume that $Q$ does not depend on the angular focus of the walker in  position $x$, then the model   in \cite{CPT11} can be written as
 $$Q[m(t)](x)=\int_{\R^d} D k(x-y) m(y,t) dy , \quad k(x)=\phi(|x|) 
  $$
 with $\phi\in Lip([0,+\infty))$, decreasing in $(0,r)$, increasing in $(r,R)$, and constant in $[R,+\infty)$, so with a behavior  similar to the Morse kernel \eqref{k3} and to \eqref{k2}.  If we take $\phi\in C^2((0,+\infty))$ with $\phi''$ bounded, then $F$ given by $\eqref{F}$ satifies the assumptions of Theorem \ref{theo_main1}.

\subsubsection{On uniqueness of solutions}
If we assume in addition that $k\in C^2(\R^d)$ with $D^2k$ bounded, then the limit equation \eqref{AE} has a unique solution with compact support in space, as observed in Remark \ref{cor:uniq}. This occurs, for instance, in Section \ref{crowd} if $\phi\in C^2([0,+\infty))$ and $\phi'(0)=\phi''(0)=0$ (recall that   $k(x)=\phi(|x|)$). Uniqueness is also  known for the aggregation equation with kernels like those of  Section \ref{aggreg}: see \cite{BV06,CR} and the references therein. However, we  expect uniqueness of solutions to the Mean-Field Game system \eqref{mfgIntro} with $F$ given by $\eqref{F}$ only for the exponential kernel \eqref{k1}, and not in all other models where there is attraction among individuals in some range of densities. In fact, the uniqueness of solutions in Mean Field Games is strongly connected with a property of mononicity of $F$ discovered by Lasry and Lions \cite{LL07mf}. For coupling functions of the form $\eqref{F}$ such monotonicity is equivalent to the property that $k$ is a positive semidefinite kernel, namely,
\[
\int_{\R^d} \int_{\R^d}  k(x-y)v(y)v(x)\, dy dx \geq 0   \qquad \forall \, v .
\]
This property is deeply studied and has several characterizations. If $k(x-y)=\psi(|x-y|^2)$ with $\psi \in C^\infty((0,+\infty))$ and continuous in $0$, then it is known that $k$ is a positive semidefinite kernel if and only if $\psi$ is completely monotone, namely, $\psi'\leq 0$ and all other derivatives have alternating signs \cite{Fass}. In all examples describing attraction it occurs that $\phi'$, and therefore $\psi'$, is instead positive in some range. Then the MFG is not expected to have a unique solution and our result also says that the distance among the possibly multiple solutions of the MFG system tends to 0 as $\l$ becomes large.

\section{Convergence for some MFGs  of acceleration towards the Cucker-Smale model}
\label{basic2BIS}

For $\l >0$ and $0<T<+\infty$, we now consider the MFG systems of acceleration, which is written in a formal way as:
\be
\label{mfglam2BIS}
\left\{ \begin{array}{l}
\ds -\partial_t u_\l +\l u_\l -v\cdot D_xu_\l +\frac\l 2 |D_vu_\lambda|^2 = F(x,v, m_\lambda(t)) \qquad {\rm in }\; \R^{2d}\times (0,T)
\\
\ds \partial_t m_\lambda 
+ v\cdot D_xm_\l -\dive_v(m_\lambda 
\l D_v u_\lambda)=0 \qquad {\rm in }\; \R^{2d}\times (0,+\infty)\\
\ds m_\lambda(0)=m_0,\; u_\l(x,v,T)=0 \qquad {\rm in }\;  \R^{2d} .
\end{array}\right.
\ee
Here the  space variables are  denoted by $(x,v)$, with $(x,v)\in \R^d\times \R^d$.  System \eqref{mfglam2BIS} models a Nash equilibrium of a game in which the (small) players, given the flow $(m_\l(t))$ of probability measures on $\R^{2d}$, try to minimize over $\gamma$ the quantity 
$$
\int_0^T e^{-\l t} \left(\frac{1}{2\l}|\ddot \gamma(t)|^2 + F(\gamma(t), \dot \gamma(t), m_\l(t))\right) dt, 
$$
while the flow $(m_\l(t))$ is the evolution of the  positions and the velocities of the players when they play in an optimal way. 

We assume that the coupling function $F$ is a cost associated to the Cucker-Smale model:
\be
\label{CSBIS}
F(x,v,m(t)) = k * m(x,v,t) = \int_{\R^{2d}}  k(x-y, v-v_*)m(y,v_*,t) dy dv_* , \quad k(x,v)=\frac{|v|^2}{g(x)} ,
\ee
where $g:\R^d\to \R$ is bounded below by a positive constant, is even, smooth and such that $|Dg|/g$ is globally bounded. For instance,
\be
\label{CS2BIS}
g(x) = (\a +|x|^2)^\b, \quad  \a>0 , \beta \geq 0 .
\ee
In this case $D_vk(x,v)=\frac{2v}{g(x)}$ and so
$$D_v F(x,v,m(t)) = (D_vk)* m (x,v,t) = \int_{\R^{2d}}  2\frac{( v-v_*)}{g(x-y)}m(y,v_*,t) dy dv_*.$$

The aim of this section is to show that $m_\lambda\to m$ as $\lambda\to +\infty$, where $m$ solves the continuous version of the Cucker-Smale model:  
\be
\label{limitFPBIS}
\left\{ \begin{array}{l}
\ds \partial_t m  + v\cdot D_xm -\dive_v(m D_vF(x,v, m))=0 \qquad {\rm in }\; \R^{2d}\times (0,+\infty),\\
\ds m(0)=m_0,\qquad {\rm in }\;  \R^{2d}.
\end{array}\right.
\ee

\subsection{The convergence result}

Throughout this section, we assume that $m_0$ and $F$ satisfy the following conditions: 
\begin{align}\label{Hypm0BIS}
m_0\in {\mathcal P}(\R^{2d})\; \mbox{\rm has a compact support}, 
\end{align}
and 
\begin{align}\label{HypFBIS}
\mbox{\rm $F$ is given by \eqref{CSBIS} where $g:\R^d\to \R$ is bounded below by a positive constant,}\\
\mbox{\rm is even, smooth, and $|Dg|/g$ is globally bounded.} \notag
\end{align}

Let us start by describing what we mean by a weak (variational) solution of the MFG problem. Let $\Gamma=C^1([0,T],\R^d)$ endowed with usual $C^1$ norm and ${\mathcal P}(\Gamma)$ be the set of Borel probability measures on $\Gamma$. We consider, for $\eta\in {\mathcal P}(\Gamma)$, 
$$
\J_\l(\eta) = \int_{\Gamma} \int_0^{T} e^{-\l t} \frac{1}{2\l} |\ddot \gamma(t)|^2dt\eta(d\gamma)+  \int_0^{T} e^{-\l t}{\mathcal F}(m^\eta(t))dt, 
$$
where $m^\eta(t)=\tilde e_t\sharp \eta$ (with $\tilde e_t:\Gamma\to \R^{2d}$, $\tilde e_t(\gamma)=(\gamma(t), \dot \gamma(t))$) and 
$$
{\mathcal F}(m) = \frac12\int_{\R^{4d}} k(x-x_*,v-v_*) m(dx,dv)m(dx_*,dv_*)\qquad \forall m\in {\mathcal P}(\R^{2d}).
$$

\begin{Lemma}\label{ExistenceMFGa} For any $\l>0$, there exists at least a minimizer $\bar \eta_\l$ of $\J_\l$ under the constraint $\tilde e_0\sharp \bar \eta_\l = m_0$. It is a weak solution of the MFG problem of acceleration, in the sense that, for $\bar \eta_\l-$a.e. $\bar \gamma\in \Gamma$,  
\begin{align}\label{pbbarg}
& \int_0^{T} e^{-\l t} (\frac{1}{2\l} |\ddot{\bar \gamma}(t)|^2+ F(\bar \gamma(t), \dot{\bar \gamma}(t),  m^{\bar \eta_\l}(t)))dt \\
& \qquad \qquad  = \inf_{\gamma\in H^2, \; (\gamma(0), \dot \gamma(0))=(\bar \gamma(0), \dot{\bar \gamma}(0))} 
 \int_0^{T} e^{-\l t} (\frac{1}{2\l} |\ddot  \gamma(t)|^2+ F( \gamma(t),  \dot \gamma(t),  m^{\bar \eta_\l}(t)))dt.\notag
 \end{align}
\end{Lemma}

The link between the equilibrium condition \eqref{pbbarg} and  the MFG system \eqref{mfglam2BIS} is the following: if we set 
$$
u_\l(x,v,s)=  \inf_{(\gamma(s),\dot \gamma(s))=(x,v)} 
 \int_s^{T} e^{-\l (t-s)} (\frac{1}{2\l} |\ddot  \gamma(t)|^2+ F( \gamma(t),  \dot \gamma(t),  m^{\bar \eta_\l}(t)))dt,
$$
then the pair $(u_\l, m^{\bar \eta_\l})$ is (at least formally) a weak solution of \eqref{mfglam2BIS}, in the sense that $u_\l$ is a viscosity solution to the first equation in \eqref{mfglam2BIS} while $m^{\bar \eta_\l}$ is a solution in the sense of distribution of the second equation in \eqref{mfglam2BIS}. Existence of a solution to the equilibrium condition \eqref{pbbarg} for more general MFG systems is obtained in \cite{CM}, however under a much more restrictive growth condition on $F$. In addition, \cite{AMMT, CM} show that there exists a weak solution to the MFG system of acceleration \eqref{mfglam2BIS}. 

We postpone the (quite classical) proof of Lemma \ref{ExistenceMFGa} to the next section and proceed with the notion of solution for the kinetic equation \eqref{limitFPBIS}. Following \cite{CCR}, we say a map $m\in C^0([0,T],{\mathcal P}_2(\R^d))$ is a measure-valued solution to \eqref{limitFPBIS} if $m(t)= P^{x,v}(t)\sharp m_0$ where  $P^{x,v}(t)=(P^{x,v}_1(t),P^{x,v}_2(t))\in \R^{d}\times \R^d$ solves the ODE 
\be\label{defP}
\left\{\begin{array}{l}
\frac{d}{dt} P_1^{x,v}(t)= P_2^{x,v}(t),\\
\frac{d}{dt} P_2^{x,v}(t)= -D_vF(P_1^{x,v}(t), P_2^{x,v}(t), m(t)),\\
P^{x,v}(0)=(x,v).
\end{array}\right.
\ee
In \cite{CCR}, the authors propose several conditions under which such a measure-valued solution exists and is unique. This include the case of the Cucker-Smale model studied here, under the assumption that $m_0$ has a compact support. 

Our main result is the following: 

\begin{Theorem}\label{thm.mainBIS} Let $\bar \eta_\l$ be a minimizer of $\J_\l$ under the constraint $\tilde e_0\sharp \bar \eta_\l=m_0$. Then $(m^{\bar \eta_\l})$ converges as $\l\to +\infty$ to the unique measure-valued solution to  \eqref{limitFPBIS} in $C^0_{loc}([0,T), {\mathcal P}_2(\R^{2d}))$. 
\end{Theorem}

\begin{Remark}{\rm Note that we do not prove the convergence of all the equilibria $(\bar \eta_\l)$ of \eqref{pbbarg}, but only for the minimizers of $\J_\l$. The reason is that we were not able to obtain enough estimates for the other equilibria. 
}\end{Remark}

\subsection{Proof of the convergence result}

Before starting the proof, let us note that, by our assumptions, there is a constant $C_0>0$ such that 
\be\label{ineqC0}
g \geq
C_0^{-1},
 \qquad 0\leq F\leq C_0(1+|v|^2+ M_{2,v}(m)), \qquad {\rm where}\; M_{2,v}(m):= \int_{\R^{2d}}|v|^2 m(dx,dv).
\ee
\be\label{bdDF}
|D_xF(x,v,m)|\leq C_0F(x,v,m), \qquad |D_vF(x,v,m)|\leq C_0F^{1/2}(x,v,m).
\ee
 Indeed, 
$$
|D_xF(x,v,m)|\leq \int_{\R^{2d}} |Dg(x-x_*)| \frac{|v-v_*|^2}{(g(x-x_*))^2} m(dx_*,dv_*) \leq   \|Dg/g\|_\infty F(x,v,m), 
$$
while,  as $g\geq c$ (for some $c>0$),
\begin{align*}
& |D_vF(x,v,m)|  \leq \int_{\R^{2d}} \frac{2|v-v_*|}{g(x-x_*)} m(dx_*,dv_*) \\
&\qquad  \leq \Bigl(  \int_{\R^{2d}} \frac{|v-v_*|^2}{g(x-x_*)} m(dx_*,dv_*)\Bigr)^{1/2}
\Bigl( \int_{\R^{2d}} \frac{4}{g(x-x_*)} m(dx_*,dv_*)\Bigr)^{1/2} \leq 2 
 c^{-1/2} 
F^{1/2}(x,v,m). 
\end{align*}
Throughout the proof (and unless specified otherwise), $C$ denotes a constant which may vary from line to line and depends only on $T$, $d$, $m_0$ and the constant $C_0$ in \eqref{ineqC0} and \eqref{bdDF}. 

Let us now explain the existence of a minimizer for $\J_\l$. 

\begin{proof}[Proof of Lemma \ref{ExistenceMFGa}] 
Let $\ep>0$ and $\eta_\ep$ be $\ep-$optimal in Problem \eqref{pbbarg}. We define $\eta\in {\mathcal P}(\Gamma)$ by  
$$
\int_\Gamma \phi(\gamma)\eta(d\gamma) = \int_{\R^{2d}} \phi( t\to x+tv) m_0(dx,dv) \qquad \forall \phi\in C^0_b(\Gamma).
$$
Let $\pi_2:\R^{2d}\to \R^d$ defined by $\pi_2(x,v)=v$. Then $\pi_2\sharp m^\eta(t)= \pi_2\sharp m_0$ for any $t\in [0,T]$ because, for any $\phi\in C^0_b(\R^d)$, 
\begin{align*}
\int_{\R^d} \phi(v) \pi_2\sharp m^\eta(dv,t)& = \int_{\Gamma}  \phi(\dot \gamma(t))\eta(d\gamma)=\int_{\R^{2d}}  \phi(\frac{d}{dt}(t\to x+tv)) m_0(dx,dv)= 
\int_{\R^{d}}  \phi(v) \pi_2\sharp m_0(dv).
\end{align*}
Hence, by $\ep-$optimality of $\eta_\ep$,   
$$
\J_{\l} ( \eta_\ep) \leq \ep+ \J_\l(\eta) = \ep+ \int_{\Gamma}  \int_0^{T} e^{-\l t}{\mathcal F}(m^\eta(t))dt, 
$$
where, for any $t\geq 0$, and as $\pi_2\sharp m^\eta(t)= \pi_2\sharp m_0$, 
$$
{\mathcal F}(m^\eta(t)) \leq C_0 \int_{\R^{2d}} |v-v_*|^2 m^\eta(x,v,t) m^\eta(x_*,v_*,t) \leq 2C_0 \int_{\R^{2d}} |v|^2 m^\eta(x,v,t) = 2 C_0  M_{2,v}(m_0). 
$$
This shows that 
$$
\J_{\l} (\eta_\ep) \leq  \ep+ 2 \l^{-1}  C_0  M_{2,v}(m_0).
$$ 
As ${\mathcal F}$ is nonnegative, this implies that 
$$
\int_{\Gamma} \int_0^{T}e^{-\l t} \frac{1}{2\l} |\ddot \gamma(t)|^2 dt \eta_\ep(d\gamma) \leq \J_{\l} (\eta_\ep) \leq  \ep+ 2 \l^{-1}  C_0  M_{2,v}(m_0).
$$
As $m_0$ has a compact support (say contained in $B_{R_0}$) and the set 
$$
\{\gamma \in \Gamma, \; |(\gamma(0), \dot \gamma(0))| \leq R_0, \; \int_0^{T}e^{-\l t}  |\ddot \gamma(t)|^2 dt\leq C\}
$$
is compact in $\Gamma$ for any $C$, we conclude that the family $(\eta_\ep)$ is tight. By lower semi-continuity of $\J_\l$ we can then conclude that there exists a minimizer $\bar \eta_\l$ of $\J_\l$ under the (closed)  constraint $\tilde e_0\sharp \bar \eta_\l=m_0$. Note for later use that, in view of the above estimates, 
$$
\int_{\Gamma} \int_0^{T}e^{-\l t} \frac{1}{2\l} |\ddot \gamma(t)|^2 dt \bar \eta_\l(d\gamma) \leq 2 \l^{-1}  C_0  M_{2,v}(m_0), 
$$
so that, as $m_0$ has a compact support, 
\be\label{bdM2vmll}
\sup_{t\in [0,T]} M_{2,v}(m^{\bar \eta_\l}(t)) \leq C_\l,
\ee
for some constant $C_\l$ depending on $m_0$, $C_0$ and $\l$. 

Next we show equality \eqref{pbbarg}. Let $\gamma_0$ belong to the support of $\bar \eta_\l$ and set $(x_0,v_0)=(\gamma_0(0),\dot \gamma_0(0))$. Fix $\gamma_1 \in H^2([0,T], \R^d)$ with $(\gamma_1(0),\dot \gamma_1(0))=(x_0,v_0)$. For $\ep,\delta>0$, let $E_\ep=\{\gamma\in \Gamma, \; \|\gamma-\gamma_0\|_{C^1}\leq \ep\}$, $\tilde m_{\ep}= \tilde e_t\sharp( \bar \eta_\l \lfloor  E_\ep)$ and define $\eta_{\ep,\delta}$ as the Borel measure on $\Gamma$ by  
$$
\int_\Gamma \phi(\gamma) \eta_{\ep,\delta}(d\gamma) = \int_{E_\ep^c} \phi(\gamma) \bar \eta_\l(d\gamma) +
(1-\delta) \int_{E_\ep} \phi(\gamma) \bar \eta_\l(d\gamma) + \delta \int_{\R^{2d}} \phi(\gamma_1+(x-x_0+t(v-v_0)))\tilde m_{\ep}(dx,dv,0)
$$
for any $ \phi\in C^0_b(\Gamma)$. 
Let $\hat m_\ep(t)$ be the Borel measure on $\R^{2d}$ defined by 
$$
\int_{\R^{2d}} \phi(x,v)\hat m_\ep(dx,dv,t)= \int_{\R^{2d}} \phi(\gamma_1(t)+x-x_0, \dot \gamma_1(t)+v-v_0)\tilde m_{\ep}(dx,dv,0),\qquad \forall \phi\in C^0_b(\R^{2d}).
$$
We note that
\be\label{zqeulsirdigc}
m^{\eta_{\ep,\delta}}(t)=  m^{\bar \eta_\l}(t)+ \delta (\hat m_\ep(t)-\tilde m_\ep(t)), \qquad m^{\eta_{\ep,\delta}}(0)=m_0.
\ee
Hence, testing the optimality of $\bar \eta_\l$ for $\J_\l$ against  $\eta_{\ep,\delta}$ and using the definition of $\eta_{\ep,\delta}$, we obtain 
\begin{align*}
& \delta \int_{E_\ep} \int_0^T e^{-\l t}\frac{1}{2\l}|\ddot \gamma(t)|^2 dt \bar \eta_\l(d\gamma) +\int_0^T e^{-\l t} {\mathcal F}(m^{\bar \eta_\l}(t))dt \\
& \qquad \leq     \delta (\int_{\R^{2d}} \tilde m_\ep(dx,dv,0)) (\int_0^T e^{-\l t}\frac{1}{2\l}|\ddot \gamma_1(t)|^2 dt)+  \int_0^T e^{-\l t} {\mathcal F}(m^{\eta_{\ep,\delta}}(t))dt . 
\end{align*}
By definition of ${\mathcal F}$ and the fact that $k$ is even, \eqref{zqeulsirdigc} implies that 
\begin{align*}
& \delta \int_{E_\ep} \int_0^T e^{-\l t}\frac{1}{2\l}|\ddot \gamma(t)|^2 dt \bar \eta_\l(d\gamma)  \leq    \delta (\int_{\R^{2d}} \tilde m_\ep(dx,dv,0)) (\int_0^T e^{-\l t}\frac{1}{2\l}|\ddot \gamma_1(t)|^2 dt)\\
&  +  \delta \int_0^Te^{-\l t} \int_{\R^{4d}} k(x-x_*,v-v_*)m^{\bar \eta_\l}(dx_*,dv_*,t)(\hat m_\ep-\tilde m_\ep)(dx,dv,t) dt\\
&   +\frac{\delta^2}{2}\int_0^Te^{-\l t} \int_{\R^{4d}} k(x-x_*,v-v_*)(\hat m_\ep-\tilde m_\ep)(dx_*,dv_*,t) 
(\hat m_\ep-\tilde m_\ep)(dx,dv,t) dt. 
\end{align*}
We divide by $\delta>0$ and let $\delta\to 0$ to obtain, using the definition of $F$: 
\begin{align*}
&  \int_{E_\ep} \int_0^T e^{-\l t}\frac{1}{2\l}|\ddot \gamma(t)|^2 dt \bar \eta_\l(d\gamma)  \leq     (\int_{\R^{2d}} \tilde m_\ep(dx,dv,0)) (\int_0^T e^{-\l t}\frac{1}{2\l}|\ddot \gamma_1(t)|^2 dt)\\
& \qquad +   \int_0^Te^{-\l t}\int_{\R^{2d}} F(x,v,m^{\bar \eta_\l}(t)) (\hat m_\ep(dx,dv,t)-\tilde m_\ep(dx,dv,t))dt. 
\end{align*}
Rearranging, we find by the definition of $\tilde m_\ep$ and $\hat m_\ep$: 
\begin{align}\label{larejznrdfc}
&  \int_{E_\ep} \int_0^T e^{-\l t}(\frac{1}{2\l}|\ddot \gamma(t)|^2 + F(\gamma(t), \dot \gamma(t), m^{\bar \eta_\l}(t))) dt \ \bar \eta_\l(d\gamma) \\
&\qquad  \leq    
 \int_{\R^{2d}} \int_0^T e^{-\l t}(\frac{1}{2\l}|\ddot \gamma_1(t)|^2 + F(\gamma_1(t)+x-x_0, \dot \gamma_1(t)+v-v_0, m^{\bar \eta_\l}(t))) dt\ \tilde m_\ep(dx,dv,0)
\end{align}
Fix $\kappa>0$ small. By lower-semicontinuity on $\Gamma$ of the functional 
$$
\gamma \to \int_0^T e^{-\l t}(\frac{1}{2\l}|\ddot \gamma(t)|^2 + F(\gamma(t), \dot \gamma(t), m^{\bar \eta_\l}(t))) dt,
$$ 
we have, for any $\ep>0$ small enough, that, for any $\gamma\in E_\ep$,
\begin{align*}
& \int_0^T e^{-\l t}(\frac{1}{2\l}|\ddot \gamma_0(t)|^2 + F(\gamma_0(t), \dot \gamma_0(t), m^{\bar \eta_\l}(t))) dt \\
&\qquad\qquad\qquad  \leq \int_0^T e^{-\l t}(\frac{1}{2\l}|\ddot \gamma(t)|^2 + F(\gamma(t), \dot \gamma(t), m^{\bar \eta_\l}(t))) dt+\kappa.
\end{align*}
On the other hand, by the regularity of $F$ in \eqref{bdDF} and the bound on $M_{2,v}(m^{\bar \eta_\l}(t))$ in \eqref{bdM2vmll}, we have, for $|(x,v)|\leq \ep$ and $\ep\in (0,1)$, 
\begin{align*}
& \int_0^T e^{-\l t} F(\gamma_1(t)+x-x_0, \dot \gamma_1(t)+v-v_0, m^{\bar \eta_\l}(t)) dt \\
&\qquad  \leq  \int_0^T e^{-\l t} F(\gamma_1(t), \dot \gamma_1(t), m^{\bar \eta_\l}(t)) dt+ C(\gamma_1,\l)\ep. 
\end{align*}
Plugging the inequalities above into \eqref{larejznrdfc} gives 
\begin{align*}
& \bar \eta_\l (E_\ep)\Bigl(  \int_0^T e^{-\l t}(\frac{1}{2\l}|\ddot \gamma_0(t)|^2 + F(\gamma_0(t), \dot \gamma_0(t), m^{\bar \eta_\l}(t))) dt+ \kappa \Bigr)  \\
&\qquad  \leq    
(\int_{\R^{2d}}\tilde m_\ep(dx,dv,0)) \Bigl(\int_0^T e^{-\l t}(\frac{1}{2\l}|\ddot \gamma_1(t)|^2 + F(\gamma_1(t), \dot \gamma_1(t), m^{\bar \eta_\l}(t))) dt+ C(\gamma_1,\l)\ep\Bigr). 
\end{align*}
As $\bar \eta_\l (E_\ep)= (\int_{\R^{2d}}\tilde m_\ep(dx,dv,0))$, we can divide the inequality above by this quantity (which is positive since $\gamma_0$ is in the support of $\bar \eta_\l$) and then let $\ep\to 0$, $\kappa\to0$ to obtain
\begin{align*}
&  \int_0^T e^{-\l t}(\frac{1}{2\l}|\ddot \gamma_0(t)|^2 + F(\gamma_0(t), \dot \gamma_0(t), m^{\bar \eta_\l}(t))) dt  \\
&\qquad  \leq    \int_0^T e^{-\l t}(\frac{1}{2\l}|\ddot \gamma_1(t)|^2 + F(\gamma_1(t), \dot \gamma_1(t), m^{\bar \eta_\l}(t))) dt, 
\end{align*}
which gives \eqref{pbbarg}. 
\end{proof}

From now on we fix $\bar \eta_\l$ a minimizer of $\J_\l$ under the constraint $\tilde e_0\sharp \bar \eta_\l=m_0$ and set 
$$
u_\l(x,v,s)=  \inf_{\gamma\in H^2, (\gamma(s),\dot \gamma(s))=(x,v)} 
 \int_s^{T} e^{-\l (t-s)} (\frac{1}{2\l} |\ddot  \gamma(t)|^2+ F( \gamma(t),  \dot \gamma(t),  m^{\bar \eta_\l}(t)))dt.
 $$
 
We now note that this value function is bounded: 
 
 \begin{Lemma}\label{lem.bds} We have 
 \be\label{oaisjdg}
 \J_\l(\bar \eta_\l) \leq 2C_0\l^{-1}  M_{2,v}(m_0), 
 \ee
and, for any $0\leq s\leq t\leq T$, 
\begin{align}\label{lzaekj:nrsfd}
M_{2,v}(m^{\bar \eta_\l}(t))= & \int_{\R^{2d}} |v|^2 m^{\bar \eta_\l}(dx,dv,t) \leq  2 (1+  4C_0 \l^{-1}  e^{\l (t-s)})  M_{2,v}(m^{\bar \eta_\l}(s))
\end{align} 
 and 
 $$
0\leq u_\l(x,v,s)\leq C\l^{-1}(1+|v|^2+M_{2,v}(m^{\bar \eta_\l}(s))). 
$$
\end{Lemma}

\begin{Remark}{\rm 
We use here the fact that we work in a finite horizon problem to obtain the last inequality from \eqref{lzaekj:nrsfd}: see the end of the proof. 
}\end{Remark}

\begin{proof} The key point of the proof consists in  refining the estimate \eqref{bdM2vmll} obtained in the proof of Lemma \ref{ExistenceMFGa}. For this we need to introduce a few notations. Given $s\in [0,T)$, let $\Gamma_s=C^1([s,T],\R^d)$ and,  for $\eta\in {\mathcal P}(\Gamma_s)$, 
$$
\J_{\l,s}(\eta) = \int_{\Gamma_s} \int_s^{T} e^{-\l (t-s)} \frac{1}{2\l} |\ddot \gamma(t)|^2dt\eta(d\gamma)+  \int_s^{T} e^{-\l (t-s)}{\mathcal F}(m^\eta(t))dt, 
$$
By dynamic programming principle (see Lemma \ref{lem.DPP} below), the restriction $\bar \eta_{\l,s}$ of $\bar \eta_\l$ defined by 
$$
\int_{\Gamma_s}\phi(\gamma)\bar \eta_{\l,s}(d\gamma)= \int_{\Gamma}\phi(\gamma_{|_{[s,T]}})\bar \eta_{\l}(d\gamma)\qquad \forall \phi\in C^0_b(\Gamma_s), 
$$
is a minimizer of $\eta\to \J_{\l,s}(\eta)$ under the constraint $\tilde e_s\sharp \eta= \tilde e_s \sharp \bar \eta_\l$. 

Defining $\eta\in {\mathcal P}(\Gamma_s)$ by  
$$
\int_{\Gamma_s} \phi(\gamma)\eta(d\gamma) = \int_{\R^{2d}} \phi( t\to x+tv) m^{\bar \eta_\l}(dx,dv,s) \qquad \forall \phi\in C^0_b(\Gamma_s), 
$$
we obtain 
$$
\J_{\l,s} (\bar \eta_{\l,s}) \leq \J_\l(\eta)=  \int_{\Gamma}  \int_s^{T} e^{-\l (t-s)}{\mathcal F}(m^\eta(t))dt, 
$$
where, as in the proof of Lemma \ref{ExistenceMFGa}, for any $t\geq s$, 
$$
{\mathcal F}(m^\eta(t)) \leq C_0 \int_{\R^{2d}} |v-v_*|^2 m^\eta(x,v,t) m^\eta(x_*,v_*,t) \leq 2 C_0  M_{2,v}(m^{\bar \eta_\l}(s)). 
$$
This shows that 
\be\label{lzkjaensfdx}
\J_{\l,s} (\bar \eta_{\l,s}) \leq  2 \l^{-1}  C_0  M_{2,v}(m^{\bar \eta_\l}(s))
\ee
and inequality \eqref{oaisjdg} holds if we choose $s=0$. 

Next we note that $M_{2,v}(m^{\bar \eta_\l}(t))$  is finite: we have, for $\bar \eta_\l-$a.e. $\bar \gamma$, and any $0\leq s\leq t\leq T$, 
\begin{align*}
|\dot{\bar \gamma}(t)-\dot{\bar \gamma}(s)| & \leq \Bigl(\int_s^t e^{-\l (\tau-s)} |\ddot{\bar \gamma}(\tau)|^2d \tau\Bigr)^{1/2}\Bigl( \int_s^t e^{\l (\tau-s)} d\tau\Bigr)^{1/2}, 
\end{align*} 
so that { (by the elementary inequality $a^2-2b^2\leq2|a-b|^2$),}
\begin{align*}
|\dot{\bar \gamma}(t)|^2\leq 2|\dot{\bar \gamma}(s)|^2+ 2\l^{-1} e^{\l (t-s)}  \Bigl(\int_s^t e^{-\l (\tau-s)} |\ddot{\bar \gamma}(\tau)|^2d\tau\Bigr). 
\end{align*} 
Integrating with respect to $\bar \eta_{\l,s}$ gives, using \eqref{lzkjaensfdx} { in the last inequality},
\begin{align*}
& \int_{\R^{2d}} |v|^2 m^{\bar \eta_\l}(dx,dv,t)  = \int_\Gamma |\dot{\bar \gamma}(t)|^2\bar \eta_\l(d\bar \gamma)= \int_{\Gamma_s} |\dot{\bar \gamma}(t)|^2\bar \eta_{\l,s}(d\bar \gamma)\\
& \qquad\qquad \leq 2 \int_{\Gamma_s} |\dot{\bar \gamma}(s)|^2\bar \eta_{\l,s}(d\bar \gamma)
+ 2\l^{-1} e^{\l (t-s)}  \int_{\Gamma_s} \int_s^{T} e^{-\l (\tau-s)} |\ddot{\bar \gamma}(\tau)|^2d\tau \bar \eta_{\l,s}(d\bar \gamma)\\
& \qquad\qquad \leq 2 \int_{\R^{2d}} |v|^2 m^{\bar \eta_\l}(dx,dv,s) + 4 e^{\l (t-s)} \J_{\l,s} (\bar \eta_{\l,s}) \\
& \qquad\qquad  \leq  2 (1+  4 C_0 \l^{-1}  e^{\l (t-s)})  M_{2,v}(m^{\bar \eta_\l}(s)).
\end{align*} 
This proves \eqref{lzaekj:nrsfd}. Finally, using $\gamma(t)=x+(t-s)v$ as a test function for $u_\l(x,v,s)$, we have:
\begin{align*}
u_\l(x,v,s) & \leq \int_s^{T}e^{-\l (t-s)} F(x+(t-s)v, v, m^{\bar \eta_\l}(t))dt \\
& \leq \int_s^T e^{-\l (t-s)} C_0(1+ |v|^2 + M_{2,v}(m^{\bar \eta_\l}(t)))dt,
\end{align*} 
which gives the result thanks to \eqref{lzaekj:nrsfd}. Note that if we were working with an infinite horizon problem, the right-hand side of the inequality above could be unbounded. 
\end{proof}

\begin{Lemma}\label{lem.DPP} Under the notation of the proof of Lemma \ref{lem.bds} and  for any $s\in [0,T)$, $\bar \eta_{\l,s}$ is a minimizer of $\eta\to \J_{\l,s}(\eta)$ under the constraint $\tilde e_s\sharp \eta= \tilde e_s \sharp \bar \eta_\l$. 
\end{Lemma}

\begin{proof} Let us set, for $m\in {\mathcal P}(\R^{2d})$ and $s\in [0,T)$,  
$$
{\mathcal V}_\l(m,s)= \inf\{ \J_{\l,s}(\eta), \;\eta\in {\mathcal P}(\Gamma_s), \; \tilde e_s\sharp \eta = m\} . 
$$
We claim that 
\begin{align}
{\mathcal V}_\l(m_0,0)& = \inf_{\eta\in  {\mathcal P}(\Gamma), \tilde e_0\sharp \eta=m_0} \int_0^se^{-\l \tau}( \int_{\Gamma}\frac{1}{2\l} |\ddot \gamma(\tau)|^2 \eta(d\gamma)+{\mathcal F}(m^\eta(\tau)))d\tau +e^{-\l s} {\mathcal V}_\l(m^{\eta}(s), s) \notag\\
&= \int_0^se^{-\l \tau}( \int_{\Gamma}\frac{1}{2\l} |\ddot\gamma(\tau)|^2 \bar \eta_\l(d\gamma)+{\mathcal F}(m^{\bar\eta_\l}(\tau)))d\tau + e^{-\l s}{\mathcal V}_\l(m^{\bar \eta_\l}(s), s), \label{zakzerjntd}
\end{align}
which proves the lemma. The proof of \eqref{zakzerjntd} is a straightforward application of the usual techniques of dynamic programming, the only point being to be able to concatenate at time $s$ two measures $\eta_1\in {\mathcal P}(\Gamma)$ and $\eta_2\in {\mathcal P}(\Gamma_s)$ such that $m:=\tilde e_s\sharp\eta_1= \tilde e_s\sharp \eta_2$. For this, let us denote by 
$\gamma_1\wedge \gamma_2$ (for $\gamma_1\in \Gamma$ and $\gamma_2\in \Gamma_s$ such that $(\gamma_1(s), \dot \gamma_1(s))= (\gamma_2(s), \dot \gamma_2(s))$) the map in $\Gamma$ such that 
$$
\gamma_1\wedge \gamma_2(t)= \left\{\begin{array}{ll} \gamma_1(t) & {\rm if} \; t\in [0,s],\\ \gamma_2(t) & {\rm if} \; t\in[s,T].\end{array}\right. 
$$
In order to define the concatenation $\eta_1\wedge \eta_2$, we disintegrate  $\eta_1$ (respectively $\eta_2$) with respect to the measure $m$. We have 
$$
\eta_1(d\gamma) = \int_{\R^{2d}} \eta_{1,x,v} (d\gamma)m(dx,dv) \qquad ({\rm resp.}\; \eta_2(d\gamma) = \int_{\R^{2d}} \eta_{2,x,v} (d\gamma)m(dx,dv)),
$$
where for $m-$a.e. $(x,v)$ and for $(\eta_{1,x,v}+\eta_{2,x,v})-$a.e. $\gamma$, one has $(\gamma(s),\dot\gamma(s))=(x,v)$. We then define $\eta_1\wedge\eta_2\in {\mathcal P}(\Gamma)$ by  
$$
\int_{\Gamma}\phi(\gamma) (\eta_1\wedge \eta_2)(d\gamma)=  \int_{\R^{2d}}  \int_{\Gamma\times \Gamma_s} \phi(\gamma_1\wedge \gamma_2) \eta_{1,x,v}(d\gamma_1)\eta_{2,x,v}(d\gamma_2)m(dx,dv)\qquad \forall \phi\in C^0_b(\Gamma). 
$$
By construction we have $m^{\eta_1\wedge\eta_2}(t)= m^{\eta_1}(t)$ if $t\in  [0,s]$, $m^{\eta_1\wedge\eta_2}(t)= m^{\eta_2}(t)$ if $t\in  [s,T]$ and  
\begin{align*}
& \int_0^Te^{-\l \tau}( \int_{\Gamma}\frac{1}{2\l} |\ddot \gamma(\tau)|^2 (\eta_1\wedge\eta_2)(d\gamma)+{\mathcal F}(m^{\eta_1\wedge \eta_2}(\tau)))d\tau\\
& \qquad = \int_0^se^{-\l \tau}( \int_{\Gamma}\frac{1}{2\l} |\ddot \gamma_1(\tau)|^2 \eta_1 (d\gamma_1)+{\mathcal F}(m^{\eta_1}(\tau)))d\tau\\
& \qquad\qquad + e^{-\l s} \int_s^Te^{-\l (\tau-s)}( \int_{\Gamma_s}\frac{1}{2\l} |\ddot \gamma_2(\tau)|^2 \eta_2(d\gamma_2)+{\mathcal F}(m^{\eta_2}(\tau)))d\tau.
\end{align*}
The rest of the proof of \eqref{zakzerjntd} follows then the usual arguments of dynamic programming. 
\end{proof}

As $u_\l$ is the value function of an optimal control problem with smooth (in space) coefficients, it is locally Lipschitz continuous. We now evaluate its derivative with respect to $v$: 

\begin{Lemma}\label{lem.bdDvul} For any  $\ep>0$,  $\lambda \geq \ep^{-1}$, we have
\begin{align*}
|D_vu_\l (x,v,s)|&  \leq C_1( \l^{-1/2} u_\l^{1/2}(x,v,s) +\ep u_\l(x,v,s)) \qquad \mbox{\rm  for a.e. $(x,v,s) \in\R^{2d}\times [0,T-\ep]$},  
\end{align*}
where {$C_1=C_0+4$}.
\end{Lemma}

\begin{proof} Let $\ep>0$, $(x,v,s)$ be a point of differentiability of $u_\l$ with $s\in [0,T-\ep]$. 
Let $z^\ep:[0,+\infty)\to \R$ be defined by $z^\ep(t)= t- \frac{2t^2}{\ep}+ \frac{t^3}{\ep^2}$ on $[0,\ep]$ and $z^\ep(t)=0$ on $[\ep,+\infty)$. Then $z^\ep(0)=z^\ep(\ep)=\dot z^\ep(\ep)=0$, $\dot z^\ep(0)=1$ and $z^\ep\in H^2([0,+\infty))$. Therefore, if $\bar \gamma$ is optimal for $u_\l(x,v,s)$, we have, for any $h\in \R^d$ and using $t\to \bar \gamma(t) +  z^\ep(t-s)h$ as a competitor in $v_\l(x,v+h,s)$: 
\begin{align*}
& u_\l (x,v+h,s)\\
&  \leq 
 \int_s^{T} e^{-\l (t-s)} (\frac{1}{2\l} |\ddot{\bar  \gamma}(t)+\ddot z^\ep(t-s)h|^2
 + F( \bar\gamma(t)+z^\ep(t-s)h,  \dot{\bar \gamma}(t)+\dot z^\ep(t-s)h,  m^{\bar \eta_\l}(t))) dt\\
 & \leq  u_\l(x,v) + \int_s^{s+\ep} e^{-\l (t-s)} \Bigl(\frac{1}{\l} \ddot{\bar  \gamma}(t)\cdot (\ddot z^\ep(t-s) h) + \frac{1}{2\l} | \ddot z^\ep(t-s)|^2|h|^2 \\
& \qquad\qquad\qquad \qquad  +  \int_0^1 (D_xF\cdot (z^\ep(t-s)h) + D_vF\cdot (\dot z^\ep(t-s)h))d\tau \Bigr)dt
\end{align*}
where for simplicity we have omitted the argument $( \bar\gamma(t)+\tau z^\ep(t-s)h,  \dot{\bar \gamma}(t)+\tau \dot z^\ep(t-s)h,  m^{\bar \eta_\l}(t)))$ after $D_xF$
and $D_vF$. Dividing by $|h|$ and letting $h\to 0$ shows that 
\begin{align*}
|D_vu_\l (x,v,s)|&  \leq 
  \int_s^{s+\ep} e^{-\l (t-s)} \Bigl(\frac{1}{\l} |\ddot{\bar  \gamma}(t)|\ |\ddot z^\ep(t-s)| + |D_xF|\ |z^\ep(t-s)| + |D_vF|\  |\dot z^\ep(t-s)| \Bigr)dt,
\end{align*}
where, from now on, $F$, $D_xF$ and $D_vF$ have for argument $(\bar\gamma(t),  \dot{\bar \gamma}(t),m^{\bar \eta_\l}(t))$. Recalling \eqref{bdDF} and the expression of $z^\ep$ we get 
\begin{align*}
& |D_vu_\l (x,v,s)|  \leq \l^{-1} (\int_s^{s+\ep} e^{-\l (t-s)} |\ddot{\bar  \gamma}(t)|^2dt)^{1/2} (\int_s^{s+\ep} e^{-\l (t-s)} \ |\ddot z^\ep(t-s)|^2dt)^{1/2} \\
&\qquad \qquad \qquad \qquad \qquad \qquad  + C_0\ep \int_s^{s+\ep} e^{-\l (t-s)} Fdt  + C_0 \int_s^{s+\ep} e^{-\l (t-s)} F^{1/2}dt \\
& \qquad \leq 
 \frac 1\l \left(\frac{{16}}{\ep^2}\frac{1-e^{-\l\ep}}\l\right)^{1/2}
u_\l^{1/2}(x,v,s) +C_0\ep u_\l(x,v,s)+ C_0\l^{-1/2}u_\l^{1/2}(x,v,s).  
\end{align*}
So, if $\l \geq  \ep^{-1}$, we obtain 
\begin{align*}
|D_vu_\l (x,v,s)|&  \leq (C_0+{4}) \l^{-1/2} u_\l^{1/2}(x,v,s) +C_0\ep u_\l(x,v,s).  
\end{align*}
\end{proof}

\begin{Lemma}\label{lem.repdotgamma} Let $\bar \gamma$ be optimal for $u_\l (x,v,0)$. Then we have, for any $t\in [0,T-\ep]$,  
$$
|\ddot{\bar \gamma}(t)|\leq  2C_1 \Bigl(\l^{1/2} u_\l^{1/2}(\bar \gamma(t),\dot{\bar \gamma}(t),t) +\ep \l u_\l(\bar \gamma(t),\dot{\bar \gamma}(t),t)\Bigr),
$$
where $C_1$ is the constant in Lemma \ref{lem.bdDvul}. 
\end{Lemma}

\begin{Remark}{\rm In fact we expect that $\ddot{\bar \gamma}(t)=-\l D_vu_\l(\bar \gamma(t),\dot{\bar\gamma}(t),t)$ for any $t\in (0,T]$, which would imply the lemma (without the ``2'' in the right-hand side) thanks to Lemma \ref{lem.bdDvul}. This equality is known to hold in several frameworks \cite{BC,CS}, but we are not aware of a reference for our precise setting. The estimate in Lemma \ref{lem.repdotgamma}, much simpler to prove, suffices however for our purpose. 
}\end{Remark}

\begin{proof} As $\bar \gamma$ is a minimizer of a calculus of variation problem with smooth coefficients and with quadratic growth, it is known that $\bar \gamma$ satisfies the Euler-Lagrange equation 
 $$
 \frac{d^2}{dt^2} (\l^{-1}e^{-\l t}\ddot{\bar\gamma}_\l(t))= \frac{d}{dt}\left( e^{-\l t} D_vF(\bar\gamma_\l(t), \dot{\bar \gamma}_\l(t),m^{\bar \eta_\l}(t))\right) - e^{-\l t}D_xF(\bar\gamma_\l(t), \dot{\bar \gamma}_\l(t), m^{\bar \eta_\l}(t)). 
 $$
Therefore  $\bar \gamma$ is actually of class $H^4$ and, in particular, $C^3$. 

For $h>0$ small, let $\gamma_h(s)= \bar \gamma(t)+(s-t)\dot{\bar \gamma}(t)$. By dynamic programming principle and the optimality of $\bar \gamma$ we have: 
\begin{align} 
& u_\l(\bar \gamma(t),\dot{\bar \gamma}(t),t)\notag\\
& \qquad  = \int_t^{t+h} e^{-\l(s-t)} (\frac{1}{2\l} |\ddot{ \bar \gamma}(s)|^2+F(\bar \gamma(s),\dot{\bar \gamma}(s), m^{\bar \eta_\l}(s)))ds +e^{-\l h} u_\l(\bar \gamma(t+h),\dot{\bar \gamma}(t+h),t+h)\notag \\ 
 & \qquad \leq \int_t^{t+h} e^{-\l(s-t)} F(\gamma_h(s),\dot \gamma_h(s), m^{\bar \eta_\l}(s)))ds +e^{-\l h} u_\l( \gamma_h(t+h), \dot\gamma_h(t+h),t+h).\label{mejaznrsd}
\end{align}
Note that, by $C^3$ regularity of $\bar \gamma$, $|\bar \gamma(t+h)-\gamma_h(t+h)|\leq C_\gamma h^2$ (where, here and below, $C_\gamma $ depends here on $\gamma$ and on $\l$). So, as $u_\l$ is locally Lipschitz continuous and $\dot \gamma_h(t+h)= \dot{\bar \gamma}(t)$, we get 
\begin{align*}
& u_\l( \gamma_h(t+h), \dot\gamma_h(t+h),t+h)- u_\l(\bar \gamma(t+h),\dot{\bar \gamma}(t+h),t+h)\\
& \qquad  \leq u_\l(\bar \gamma(t+h), \dot{\bar \gamma}(t),t+h) - u_\l(\bar \gamma(t+h),\dot{\bar \gamma}(t+h),t+h)+C_\gamma h^2. 
\end{align*}
Still by $C^3$ regularity we also have  $|\dot{\bar \gamma}(t+h)-\dot{\bar \gamma}(t) -\ddot{\bar \gamma}(t)h|\leq C_\gamma h^2$.
Now the bound on $D_vu_\l$ of Lemma \ref{lem.bdDvul} yields  (setting $(x,v)= (\bar \gamma(t),\dot{\bar \gamma}(t))$)
\begin{align*}
& u_\l( \gamma_h(t+h), \dot\gamma_h(t+h),t+h)- u_\l(\bar \gamma(t+h),\dot{\bar \gamma}(t+h),t+h)\\
&\qquad \leq C_1( \l^{-1/2} u_\l^{1/2}(x,v,t) +\ep u_\l(x,v,t)) |\ddot{\bar \gamma}(t)| h+ C_\gamma h^2.
\end{align*}
Plugging this inequality into \eqref{mejaznrsd} gives, after dividing by $h$ and letting $h\to 0$,  
\begin{align*}
\frac{1}{2\l} |\ddot{ \bar \gamma}(t)|^2+F(\bar \gamma(t),\dot{\bar \gamma}(t), m^{\bar \eta_\l}(t)) & \leq F(\gamma_h(t),\dot \gamma_h(t), m^{\bar \eta_\l}(t)) \\& \qquad +C_1( \l^{-1/2} u_\l^{1/2}(x,v,t) +\ep u_\l(x,v,t)) |\ddot{\bar \gamma}(t)| .
\end{align*}
Recalling that $(\gamma_h(t), \dot\gamma_h(t))= (\bar \gamma(t),\dot{\bar \gamma}(t))$ gives the result. 
\end{proof}

\begin{Lemma}\label{lem.boundddotg} There exists $\ep_0>0$ and a constant $C>0$ such that, for any $\ep\in (0,\ep_0]$, any $\l\geq \ep^{-1}  {\vee 1}$ and any $t\in [0,T-\ep]$, the support of $m^{\bar \eta_\l}(t)$ is contained in $B_C$ and 
$$
\|\ddot{\bar \gamma}\|_{L^\infty([0,T-\ep])}\leq C \qquad {\rm for}\; \bar \eta_\l-{\rm a.e.}\; \gamma.
$$ 
In particular, $(\bar \eta_\l)$ is tight and the family $(m^{\bar \eta_\l}(t))$ is 
 relatively compact  in $C^0([0,T],{\mathcal P}_2(\R^{2d}))$. 
\end{Lemma}

\begin{proof} We have, by Lemmata \ref{lem.bds} and \ref{lem.repdotgamma}, for any $\ep>0$ and $\l\geq \ep^{-1}$, and for $\bar \eta_\l-$a.e. $\gamma$ and a.e. $t\in [0,T-\ep]$, 
\begin{align}
|\ddot{\bar \gamma}(t)|&  \leq  2C_1( \l^{1/2} u_\l^{1/2}(\bar \gamma(t),\dot{\bar\gamma}(t),t) +\l \ep u_\l(\bar \gamma(t),\dot{\bar\gamma}(t),t))  \notag \\
& \leq C(1+|\dot{\bar \gamma}(t)|+M_{2,v}^{1/2}(m^{\bar \eta_\l}({t}
))) 
+C  \ep (1+|\dot{\bar \gamma}(t)|^2+M_{2,v}(m^{\bar \eta_\l}({t}
))).  \label{liuzaesf}
\end{align} 

Let us set
$$
R_\l(t)= \inf \{r>0, \; {\rm Spt}(m^{\bar \eta_\l}(t)) \subset \R^d \times B_r\} .
$$
We note that $R_\l$ is upper semi-continuous. We now show that $R_\l$ is finite on a maximal time interval $[0,\tau_\l)$, with $\tau_\l>0$, with either $\tau_\l= T-\ep$ or $\lim_{t\to \tau_\l^-}R_\l(t)=+\infty$. {For the proof of this fact, $\l$ is fixed and all constants depend on $\l$ unless specified otherwise.} By \eqref{lzaekj:nrsfd} and \eqref{liuzaesf}, we have, for $0\leq s\leq t\leq T-\ep$ and $\bar \eta_\l-$a.e. $\bar \gamma$, 
\begin{align*}
|\ddot{\bar \gamma}(t)|&  \leq C(1+|\dot{\bar \gamma}(t)|+ \l^{-1/2}  e^{\l (t-s)/2 } M_{2,v}^{1/2}(m^{\bar \eta_\l}(s)) ) 
)\notag \\
& \qquad +C  \ep (1+|\dot{\bar \gamma}(t)|^2+ \l^{-1}  e^{\l (t-s) } M_{2,v}(m^{\bar \eta_\l}(s))). \notag 
\end{align*} 
Then, as $M_{2,v}(m^{\bar \eta_\l}(s))\leq C R_\l^2(s)$ for some 
constant $C$ depending on dimension only,
\begin{align}\label{jh;ebsrdngc}
& |\ddot{\bar \gamma}(t)|  \leq C(1+|\dot{\bar \gamma}(t)|+ \l^{-1/2}  e^{\l (t-s)/2 } R_\l(s) ) 
) \notag \\
& \qquad +C  \ep (1+|\dot{\bar \gamma}(t)|^2+ \l^{-1}  e^{\l (t-s) } R_\l^2(s)) .
\end{align} 
So, if $R_\l(s)$ is finite for some $s$  and $\l\geq 1$, $\ep\leq 1$, one can find $K$ depending only on $R_\l(s)$ and the constant $C$ in \eqref{jh;ebsrdngc} 
 such that 
\[
 |\dot{\bar \gamma}(t)|\leq |\dot{\bar \gamma}(s)|+ K  \int_0^t (1 +   |\dot{\bar \gamma}(\tau)| + \ep  |\dot{\bar \gamma}(\tau)|^2)  d\tau .
\]
Then we can compare $ |\dot{\bar \gamma}(t)|$ with the solution of the ODE
\[
\dot\phi = K(1+\phi+\ep \phi^2) , \qquad \phi(s)= |\dot{\bar \gamma}(s)| ,
\]
which is given by 
\[
\phi(t)= \Phi_\ep^{-1}\Bigl( \Phi_\ep( |\dot{\bar \gamma}(s)|)+ K(t-s)\Bigr), 
\]
where
$$
\Phi_\ep(r)= \int_0^r \frac{1}{1+\tau+\ep  \tau^2} d\tau .
$$
So one can find $\ep_0, \sigma>0$ depending only on $K$ such that, for all  $\ep\in (0,\ep_0]$,  
\[
|\dot{\bar \gamma}(t)| \leq \phi(t) \leq R_\l(s)+1 \,, \quad \forall t\in [s,s+\sigma] ,
\]  
for any $\bar \gamma\in H^2$ satisfying \eqref{jh;ebsrdngc} and $|\dot{\bar \gamma}(s)|\leq R_\l(s)$. As, by definition of $R_\l$,  $m^{\bar \eta_\l}(s)$ has a support contained in $\R^d\times B_{R_\l(s)}$, this shows that $m^{\bar \eta_\l}(t)$ has a support contained in $\R^d\times B_{R_\l(s)+1}$ for any $t\in [t,t+\sigma]$. In particular, as $m_0$ has a compact support, $R_\l(0)$ is finite and thus $R_\l(t)$ is finite at least on a small time interval $[0,\sigma]$ for some $\sigma>0$. We denote by $[0,\tau_\l)$ the maximal time interval on which $R_\l$ is finite. 
Let us assume that $\tau_\l<T-\ep$. Let $t_n\to \tau_\l^-$. If $(R_\l(t_n))$ remains bounded by a constant $M$, then by the above argument $R_\l$  is bounded by $M+1$ on $[\tau_\l,\tau_\l+\sigma]$ for some $\sigma>0$ (depending on $M$), which contradicts the definition of $\tau_\l$. Hence $\lim_{t\to \tau_\l^{-}} R_\l(t)=+\infty$. {So we have proved that $R_\l$ is finite on a maximal time interval $[0,\tau_\l)$, with $\tau_\l>0$, with either $\tau_\l= T-\ep$ or $\lim_{t\to \tau_\l^-}R_\l(t)=+\infty$.}

By definition of $m^{\bar \eta_\l}(t)$, for any $\delta>0$ and $t\in [0,\tau_\l)$ there exists $\bar \gamma\in \Gamma$ in the support of $\bar \eta_\l$ such that 
$|\dot{\bar \gamma}(t)|\geq R_\l(t)-\delta$. Thus 
\begin{align*}
R_\l(t)-\delta \leq |\dot{\bar \gamma}(t)|\leq |\dot{\bar \gamma}(0)|+ \int_0^t  |\ddot{\bar \gamma}(s)| ds.
\end{align*}
As $(\bar \gamma(t), \dot{\bar \gamma}(t))$ belongs to the support of $m^{\bar \eta_\l}(t)$ for any $t$, we get by \eqref{liuzaesf} and the definition of $R_\l$:
\begin{align*}
R_\l(t)-\delta& \leq  |\dot{\bar \gamma}(0)|+C \int_0^t (1+|\dot{\bar \gamma}(s)| + M_{2,v}^{1/2} (m^{\bar \eta_\l}(s)))ds \\
& \qquad + C\ep \int_0^t (1+|\dot{\bar \gamma}(s)|^2+M_{2,v}(m^{\bar \eta_\l}(s)))ds\\
& \leq R_0+ C\int_0^t (1+R_\l(s)+\ep R_\l^2(s))ds.
\end{align*}
As $\delta$ is arbitrary, this proves that 
\begin{align*}
R_\l(t)& \leq R_0+ C\int_0^t (1+R_\l(s)+\ep R_\l^2(s))ds\qquad \forall t\in [0, \tau_\l).
\end{align*}
Arguing as above we get
\begin{align*}
R_\l(t)\leq \Phi_\ep^{-1}\Bigl( \Phi_\ep(R_0)+ Ct\Bigr) .
\end{align*}
For all $\ep>0$ small enough (but independent of $\l$) and $\l\geq \ep^{-1}\vee 1$, we have therefore that $R_\l$ is bounded by a constant $C$ independent of $\l$ on $[0,\tau_\l)$. Thus $\tau_\l=T-\ep$ and $R_\l$ is bounded by $C$ on $[0,T-\ep]$. 

This estimate gives immediately the bound on $|\dot{\bar \gamma}|$ and therefore, by  \eqref{liuzaesf},  the bound on $|\ddot{\bar \gamma}|$ for $\bar \eta_\l-$a.e. $\bar \gamma$. As $m_0$ has a compact support, this also implies that the $m_\l(t)$ have a support contained in a ball $B_C$, where $C$ is independent of $\l$ and $t$. In addition the sequence $\bar \eta_\l$ is tight.  

Finally, we have, for any $0\leq s\leq t\leq T-\ep$,  
\begin{align*}
{\bf d}_1(m^{\bar \eta_\l}(s), m^{\bar \eta_\l}(t)) & = \int_{\Gamma} (|\bar \gamma(t)-\bar \gamma(s)|^2+ |\dot{\bar \gamma}(t)-\dot{\bar \gamma}(s)|^2)^{1/2}\bar \eta_\l(d\bar \gamma)\\
& \leq C (t-s)^{1/2} \int_{\Gamma} \Bigl(\int_s^t |\ddot{\bar \gamma}(\tau)|^2d\tau\Bigr)^{1/2} \bar \eta_\l(d\bar \gamma) \leq C (t-s)^{1/2}.
\end{align*}
 As the $(m^{\bar \eta_\l}(t))$ have a support which is uniformly bounded, this shows that it is a relatively compact sequence in $C^0([0,T],{\mathcal P}_2(\R^d))$. 
\end{proof}

We are now ready to prove the main result: 

\begin{proof}[Proof of Theorem \ref{thm.mainBIS}]
In view of Lemma  \ref{lem.boundddotg}, $(\bar \eta_\l)$ is tight and we can consider a subsequence $(\bar \eta_{\l_n})$  which converges weakly to some $\eta$ in ${\mathcal P}(\Gamma)$. Then 
 $(m^{\bar \eta_{\l_n}}(t))$ converges in $C^0([0,T],{\mathcal P}_2(\R^{2d}))$ to $m= \tilde e_t\sharp \eta$. Our aim is to prove that $m$ is a measure valued solution to the kinetic equation \eqref{limitFPBIS}. 

For this we identify the $\limsup( {\rm Spt}(\bar \eta_{\l_n}))$. Let us recall that, by Lemma \ref{ExistenceMFGa},  for $\bar \eta_\l-$a.e. $\bar \gamma_\l$, $\bar \gamma_\l$ minimizes  problem  \eqref{pbbarg}. Hence by the Euler equation we have that $\bar \gamma_\l$ is of class $H^4$ and for a.e. $t\in [0,T]$, 
 $$
 \frac{d^2}{dt^2} (\l^{-1}e^{-\l t}\ddot{\bar\gamma}_\l(t))= \frac{d}{dt}\left( e^{-\l t} D_vF(\bar\gamma_\l(t), \dot{\bar \gamma}_\l(t),m^{\bar \eta_\l}(t))\right) - e^{-\l t}D_xF(\bar\gamma_\l(t), \dot{\bar \gamma}_\l(t), m^{\bar \eta_\l}(t)). 
 $$
We rewrite this equality as 
\begin{align*}
\ddot{\bar \gamma}_\l(t)+ D_vF(\bar \gamma_\l(t), \dot{\bar \gamma}_\l(t), m^{\bar \eta_\l}(t))& = \l^{-1}\Bigl( -\l^{-1} \bar \gamma^{(iv)}_\l(t) +2\dddot{\bar \gamma}_\l(t)+
\frac{d}{dt} D_vF(\bar \gamma_\l(t), \dot{\bar \gamma}_\l(t), m^{\bar \eta_\l}(t))\\
& \qquad \qquad - D_xF(\bar \gamma_\l(t), \dot{\bar \gamma}_\l(t),m^{\bar \eta_\l}(t))\Bigr).
\end{align*}
We integrate this equation by parts against a test function $z\in C^\infty_c((0,T),\R^d)$ to get 
\begin{align*}
&\int_0^T \Bigl(-\dot{\bar \gamma}_\l(t)\cdot \dot z(t)+ D_vF(\bar \gamma_\l(t), \dot{\bar \gamma}_\l(t), m^{\bar \eta_\l}(t))\cdot z(t) \Bigr) dt \\
&\qquad = \l^{-1}\int_0^T \Bigl( \l^{-1} \dot{\bar \gamma}_\l(t)\cdot \dddot z(t) +2\dot{\bar \gamma}_\l(t)\cdot \ddot z(t)-
 D_vF(\bar \gamma_\l(t), \dot{\bar \gamma}_\l(t),m^{\bar \eta_\l}(t))\cdot \dot z(t)\\
 & \qquad\qquad\qquad \qquad - D_xF(\bar \gamma_\l(t), \dot{\bar \gamma}_\l(t),m^{\bar \eta_\l}(t))\cdot z(t)\Bigr)dt.
\end{align*}
By Lemma \ref{lem.boundddotg}
$(\bar \gamma_{\l})$ is relatively compact in $\Gamma$, and for any sequence $\l_n \to +\infty$ we can extract a subsequence such that $\bar \gamma_{\l_n} \to \gamma\in \Gamma$ and $m^{\bar \eta_{\l_n}}\to m \in C^0([0,T],{\mathcal P}_2(\R^{2d}))$. Therefore
\begin{align*}
&\int_0^T \Bigl(-\dot{ \gamma}(t)\cdot \dot z(t)+ D_vF( \gamma(t), \dot{ \gamma}(t),m(t))\cdot z(t) \Bigr) dt =0, \qquad \forall z\in C^\infty_c((0,T),\R^d),
\end{align*}
which means that it is a  solution to 
$$
\ddot \gamma(t) = -D_vF( \gamma(t), \dot{ \gamma}(t),m(t)).
$$
In other words, $(\gamma(t), \dot \gamma(t))= P^{x,v}(t)$, where $P$ is defined by \eqref{defP} and $(x,v)= (\gamma(0), \dot \gamma(0))$. 
By Lemma \ref{lem.boundddotg} we can also extract a further subsequence such that $\bar \eta_{\l_n} \rightharpoonup \eta \in {\mathcal P}(\Gamma)$. 
As the support of $\eta$ consists of solutions to \eqref{defP} and 
 $\tilde e_0\sharp \eta=m_0$, we have 
$$
\eta= \int_{\R^{2d}} \delta_{P^{x,v}}m_0(dx,dv),
$$
so that
$$
m(t)= \tilde e_t\sharp \eta= P^{x,v}(t) \sharp m_0. 
$$
Hence $m$ is the measure-valued solution to \eqref{limitFPBIS}. Following \cite{CCR} this solution is unique. We have proved therefore that any converging subsequence of the relatively compact familiy 
 $(m^{\bar \eta_\l})$ has for limit the unique solution $m$ to  \eqref{limitFPBIS}: the entire sequence converges. 
\end{proof}
\begin{Remark}\label{crowd2}
\upshape
The Cucker-Smale model is usually associated to the collective animal behaviour, such as flocking of birds or swarming of insects. However, similar models where the acceleration of the agents is prescribed have been proposed for describing the dynamics of crowds of pedestrians, and some of them 
 fit  in our results. We refer to the book \cite{CPTbook}, in particular the section on mesoscopic or kinetic models, and to the recent survey paper \cite{PR}, where they are called social forces models.
\end{Remark}

 \begin{thebibliography}{00}

\bibitem{AMMT} Achdou, Y., Mannucci, P., Marchi, C.,  Tchou, N. (2019). Deterministic mean field games with control on the acceleration. arXiv:1908.03330, to appear in	NoDEA Nonlinear Differential Equations Appl.

\bibitem{AGS} Ambrosio, L., Gigli, N., \& Savar\'e, G. (2008). {\sc Gradient flows: in metric spaces and in the space of probability measures}. Second edition. 
Birkh\"auser Verlag, Basel, 2008.


\bibitem{BC} Bardi, M.,  Capuzzo-Dolcetta, I. (2008). {\sc Optimal control and viscosity solutions of Hamilton-Jacobi-Bellman equations.} Birkh\"auser Boston, Inc., Boston, MA, 1997.

\bibitem{Ba} Barker, M. (2019). From mean field games to the best reply strategy in a stochastic framework. 
J. Dyn. Games 6 (2019),  no. 4,  291-314.
 
{ \bibitem{BeCaSa} Benamou, J. D., Carlier, G.,  Santambrogio, F. (2017). Variational mean field games. In Active Particles, Volume 1 (pp. 141-171). Birkh\"{a}user, Cham.
}

\bibitem{BT}   Bernoff, A. J.; Topaz, C. M.: Nonlocal aggregation models: a primer of swarm equilibria. SIAM Rev. 55 (2013), 
709--747

\bibitem{BGL}  Bertozzi, A. L.; Garnett, J. B.; Laurent, T.: Characterization of radially symmetric finite time blowup in multidimensional aggregation equations. SIAM J. Math. Anal. 44 (2012), 
651--681

 \bibitem{LLB} C. Bertucci, P.-L. Lions, J.-M. Lasry: Some remarks on Mean Field Games,  to appear in Comm. Partial Differential Equations.


   \bibitem{BV06}  Bodnar, M.; Velazquez, J. J. L. An integro-differential equation arising as a limit of individual cell-based models. J. Differential Equations 222 (2006), 
 341--380.
 


\bibitem{CCR} Canizo, J. A., Carrillo, J. A., Rosado, J. (2011). A well-posedness theory in measures for some kinetic models of collective motion. Mathematical Models and Methods in Applied Sciences, 21(03), 515-539.

  \bibitem{CS} Cannarsa, P., Sinestrari, C. (2004). Semiconcave functions, Hamilton-Jacobi equations, and optimal control (Vol. 58). Springer Science \& Business Media.
  
  \bibitem{CM} Cannarsa, P., Mendico, C. (2019). Mild and weak solutions of Mean Field Games problem for linear control systems. arXiv preprint arXiv:1907.02654.
   
  \bibitem{Car} P. Cardaliaguet, Notes on mean field games. Technical report.
  
{  \bibitem{Ca15} Cardaliaguet, P. (2015). Weak solutions for first order mean field games with local coupling. In Analysis and geometry in control theory and its applications (pp. 111-158). Springer, Cham.} 

{ \bibitem{CaGr15} Cardaliaguet, P.,  Graber, P. J. (2015). Mean field games systems of first order. ESAIM: Control, Optimisation and Calculus of Variations, 21(3), 690-722.
}
  
    \bibitem{CarH} P. Cardaliaguet, S. Hadikhanloo (2017). Learning in mean field games: the fictitious play. ESAIM: Control, Optimisation and Calculus of Variations, 23(2), 569-591.
    
{    \bibitem{CaMeSa} Cardaliaguet, P., M\'esz\`{a}ros, A. R., Santambrogio, F. (2016). First order mean field games with density constraints: pressure equals price. SIAM Journal on Control and Optimization, 54(5), 2672-2709.
}
    
\bibitem{CR} Carrillo, J. A., Rosado: Uniqueness of bounded solutions to aggregation equations by optimal transport methods. European Congress of Mathematics, 3-16, Eur. Math. Soc., Z\"urich, 2010.
 
      \bibitem{CPT11}   E. Cristiani, B. Piccoli, A. Tosin:  Multiscale modeling of granular flows with application to crowd dynamics. Multiscale Model. Simul. 9(1) (2011),        155-182. 
       
 \bibitem{CPTbook}  E. Cristiani, B. Piccoli, A. Tosin: Multiscale modeling of pedestrian dynamics. 
        Springer, Cham, 2014. 
    
  

\bibitem{DHL} P. Degond, M. Herty, J.G. Liu: Mean field games and model predictive control. Commun. Math. Sci. 15 (2017), 
 1403--1422.





 \bibitem{Fass} G.E. Fasshauer, Meshfree Approximations Methods with Matlab, World Scientific,
Singapore, 2007.

\bibitem{HMC} Huang, M., Malham\'e, R. P.,  Caines, P. E. (2006). Large population stochastic dynamic games: closed-loop McKean-Vlasov systems and the Nash certainty equivalence principle. Communications in Information \& Systems, 6(3), 221-252.
 
  
%

\bibitem{LL07mf}  J.-M. Lasry and P.-L. Lions, Mean field games, 
  Japanese Journal of Mathematics, 2 (2007), pp.~229--260.



{ \bibitem{LiCourse} Lions, P. L. (2010). Cours au Coll\`{e}ge de France. Available at www. college-de-france. fr.
 }
  

%
%

{\bibitem{OrPoSa} Orrieri, C., Porretta, A.,  Savar\'e, G. (2019). A variational approach to the mean field planning problem. Journal of Functional Analysis, 277(6), 1868-1957.
}   
     \bibitem{PR1}   B. Piccoli, F. Rossi,   Transport equation with nonlocal velocity in Wasserstein spaces: convergence of numerical schemes. Acta Appl. Math. 124 (2013), 73-105. 
     
    \bibitem{PR}   B. Piccoli, F. Rossi,
Measure-theoretic models for crowd dynamics
in "Crowd Dynamics Volume 1 - Theory, Models, and Safety Problems", N. Bellomo and L. Gibelli Eds, Birkhauser, 2018.

 
\bibitem{RSSS} R. Rossi, G. Savar\'e, A. Segatti, U. Stefanelli (2019). Weighted Energy-Dissipation principle for gradient flows in metric spaces. Journal de Math\'ematiques Pures et Appliqu\'ees, 127, 1-66. 
  
 \bibitem{TBL}
  Topaz, C.M.; Bertozzi, A. L.; Lewis, M. A. A nonlocal continuum model for biological aggregation. Bull. Math. Biol. 68 (2006), 
  1601--1623.

\end{document}